\documentclass[aps,pre,preprint,groupaddresses]{revtex4-1}

\usepackage{amssymb,latexsym,amsmath}
\usepackage{dsfont}

\usepackage{amsthm}
\usepackage{subfigure}
\usepackage{graphicx}
\usepackage{epstopdf}
\usepackage{setspace}
\newtheorem{proposition}{Proposition}

\newtheorem{remark}{Remark}

\newtheorem{example}{Example}

\newtheorem{claim}{Claim}

\newcommand{\refi}[1]{Figure \ref{fig:#1}}

\newcommand{\ressec}[1]{Section \ref{ssec:#1}}

\newcommand{\enma}[1]   {\ensuremath{#1}}

\newcommand{\edoc}{\end{document}}
\newcommand{\beq}{\begin{equation}}  
\newcommand{\eeq}{\end{equation}}    
\newcommand{\bea}{\begin{align}}
\newcommand{\eea}{\end{align}}
\newcommand{\beqn}{\begin{eqnarray}}
\newcommand{\eeqn}{\end{eqnarray}}
\newcommand{\beqns}{\begin{eqnarray*}}
\newcommand{\eeqns}{\end{eqnarray*}}
\newcommand{\bct}{\begin{center}}
\newcommand{\ect}{\end{center}}
\newcommand{\btmz}{\begin{itemize}}
\newcommand{\etmz}{\end{itemize}}
\newcommand{\benum}{\begin{enumerate}}
\newcommand{\eenum}{\end{enumerate}}

\newcommand{\ones}{\mathbf{1}}

\newcommand{\ip}[2]{\langle{#1},{#2}\rangle}

\newcommand{\norm}[1]{\left | #1 \right |}

\newcommand{\diag}  {\enma{\mathrm{diag}}}

\newcommand{\image}{\mbox{\rm{Im\kern.043333em}}}

\usepackage{color}
\definecolor{gray}{rgb}{0.5,0.5,0.5}

\begin{document}

\title{Synchronization of Weakly Coupled Oscillators: Coupling, Delay and Topology}
\author{Enrique Mallada and Ao Tang\\
    Cornell University, Ithaca, NY 14853
}

\singlespacing
\begin{abstract}
There are three key factors of a system of coupled oscillators that characterize the interaction among them: \emph{coupling (how to affect)}, \emph{delay (when to affect)} and \emph{topology (whom to affect)}. For each of them, the existing work has mainly focused on special cases. With new angles and tools, this paper makes progress in relaxing some assumptions of these factors. There are three main results in this paper. First, by using results from algebraic graph theory, a sufficient condition is obtained which can be used to check equilibrium stability. This condition works for arbitrary topology. It generalizes existing results and also leads to a
sufficient condition on the coupling function with which the system is guaranteed to reach synchronization. Second, it is known that identical oscillators with $\sin()$ coupling functions are guaranteed to synchronize in phase on a complete graph. Using our results, we demonstrate that for many cases certain structures instead of exact shape of the coupling function such as symmetry and concavity are the keys for global synchronization. Finally, the effect of heterogenous delays is investigated. We develop a new framework by constructing a non-delayed phase model that approximates the original one in the continuum limit. We further derive how its stability properties depend on the delay distribution. In particular, we show that heterogeneity, i.e. wider delay distribution, can help reach in-phase synchronization.
\end{abstract}

\maketitle

\section{Introduction}
The system of coupled oscillators has been widely studied in different disciplines ranging from biology \cite{Winfree1967BiologicalRhythms,Peskin75MathematicalAspects,Achermann&Kunz1999ModelingCicardian,Garciaetal2004ModelingMulticelular,YIM03} and chemistry \cite{Kiss&Zhai&Hudson2003synchronizationofchemical,York&Compton1991QuasiopticalPower} to engineering \cite{Hong&Scaglione2005AScalableSynchronization,Geoffreyetal2005FireflyinspiredSensor} and physics \cite{Marvel&Strogatz2009InvariantSubmanifold,Bressloff&Coombes99TravellingWaves}. The possible behavior of such a system can be complex. For example, the intrinsic symmetry of the network can produce multiple limit cycles or equilibria with relatively fixed phases (phase-locked trajectories) \cite{Ashwin&Swift92thedynamics}, which in many cases can be stable \cite{Ermentrout1992StablePeriodic}. Also, the heterogeneity in the natural oscillation frequency can lead to incoherence \cite{Kuramoto75LectureNotes} or even chaos \cite{Popovych&Maistrenko&Tass2005PhaseChaos}.

One particular interesting question is whether the coupled oscillators will synchronize in phase in the long run \cite{Brown&Holmes&Moehlis03globallycoupled, Jadbabaie&Motee&Barahona2004stabilityofkuramoto, Lucarelli&Wang2004decentralizedsynchronization, Monzon&Paganini2005gobalconsidarations,Mirollo&Strogatz90SynchronizationofPulseCoupled}. Besides its clear theoretical value, it also has rich applications in practice.


 In essence, there are three key factors of a system of coupled oscillators that characterize the interaction among oscillators: \emph{coupling}, \emph{delay} and \emph{topology}. For each of them, the existing work has mainly focused on special cases as explained below. In this paper, further research will be discussed on each of these three factors:
\begin{itemize}

\item Topology (\emph{whom to affect}, Section \ref{sec:main}): Current results either restrict to complete graph or a ring topology for analytical tractability~\cite{Monzon&Paganini2005gobalconsidarations}, study {\it local stability} of topology independent solutions over time varying graph~\cite{Moreau2004StabilityofContinuousTimeConsensus, Ren&Beard2004ConsensusofInformation, Stan&Sepulchre2003DissipativityCharacterization}, or introduce dynamic controllers to achieve synchronization for time-varying uniformly connected graphs~\cite{Scardovi2007,Sepulchre2008}. We develop a graph based sufficient condition which can be used to check equilibrium stability for any fixed topology. It also leads to a family of coupling function with which the system is guaranteed to reach {\it global} phase consensus for arbitrary undirected connected graph using only physically meaningful state variables.

\item Coupling (\emph{how to affect}, Section \ref{sec:coupling}): The classical Kuramoto model~\cite{Kuramoto75LectureNotes} assumes a $\sin()$ coupling function. Our study hints that certain symmetry and convexity structures are enough to guarantee global synchronization.

\item Delay (\emph{when to affect}, Section \ref{sec:delay}): {  Existing work generally assumes zero delay among oscillators or require them to be bounded up to a constant fraction of the period~\cite{Papachristodoulou&Jadbabie2006SynchronizationinOscillator}. This is clearly not satisfactory especially if the oscillating frequencies are high. We develop a new framework to study unbounded delays by constructing a non-delayed phase model that is equivalent to the original one. Using this result, we show that wider delay distribution can help reach synchronization.}

\end{itemize}

In this paper we study {\it weakly} coupled oscillators, which can be either pulse-coupled or phase-coupled. Although most of the results are presented for phase-coupled oscillators, they can be readily extended for pulse-coupled oscillators (see, e.g., \cite{Izhikevich98PhaseModels,Izhikevich99WeaklyPulseCoupled}). It is worth noticing that results in Section \ref{sec:topology&coupling} are independent of the strength of the coupling and therefore the weak coupling assumption is not necessary there.
Preliminary versions of this work has been presented in \cite{Mallada&Tang2010Synchronizationof} and \cite{Mallada&Tang10WeaklyPulseCoupled}.

The paper is organized as follows. We describe pulse-coupled and phase-coupled oscillator models, as well as their common weak coupling approximation, in Section \ref{sec:model}. Using some facts from algebraic graph theory and potential dynamics in Section \ref{sec:prelim}, we present the negative cut instability theorem in Section \ref{ssec:local_stability} to check whether an equilibrium is unstable. This then leads to Proposition \ref{prop:main} in Section \ref{ssec:global_stability} which identifies a class of coupling functions with which the system always synchronizes in phase. It is well known that the Kuramoto model produces global synchronization over a complete graph. In Section \ref{sec:coupling}, we demonstrate that a large class of coupling functions, in which the Kuramoto model is a special case, guarantee the instability of most of the limit cycles in a complete graph network. Section \ref{sec:delay} is devoted to the discussion of the effect of delay. An equivalent non-delayed phase model is constructed whose coupling function is the convolution of the original coupling function and the delay distribution. Using this approach, it is shown that sometimes more heterogeneous delays among oscillators can help reach synchronization. We conclude the paper in Section \ref{sec:conclusion}.

\section{Coupled Oscillators}
\label{sec:model}

We consider  two different models of coupled oscillators studied in the literature. The difference between the models arises in the way that the oscillators interact between each other, and their dynamics can be quite different. However, when the interactions are weak (weak coupling), both systems behave similarly and share the same approximation. This allows us to study them under a common framework.

Each oscillator is represented by a phase $\theta_i$ in the unit circle $\mathds S^1$ which in the absence of coupling moves with constant speed 
$
\dot{\theta}_i=\omega.
$
Here, $\mathds S^1$ represents the unit circle, or equivalently the interval $[0,2\pi]$ with $0$ and $2\pi$ identified ($0\equiv 2\pi$), and $\omega=\frac{2\pi}{T}$ denotes the natural frequency of the oscillation.



\subsection{Pulse-coupled Oscillators}

In this model the interaction between oscillators is performed by pulses.  An oscillator $j$ sends out a pulse whenever it crosses zero ($\theta_j=0$). When oscillator $i$ receives a pulse, it will change its position from $\theta_i$ to $\theta_i+\varepsilon\kappa_{ij}(\theta_i)$. The function $\kappa_{ij}$ represents how other oscillators' actions affect oscillator $i$ and the scalar $\varepsilon>0$ is a measure of the coupling strength. These jumps can be modeled by a Dirac's delta function $\delta$ satisfying $\delta(t)=0$ $\forall t\neq 0$, $\delta(0)=+\infty$, and $\int\delta(s)ds=1$. The coupled dynamics is represented by
\begin{equation}\label{eq:pc_model}
\dot\theta_i(t)=\omega + \varepsilon\omega\sum_{j\in\mathcal N_i}\kappa_{ij}(\theta_i(t))\delta(\theta_j(t-\eta_{ij})),
\end{equation}
where $\eta_{ij}>0$ is the propagation delay between oscillators $i$ and $j$ ($\eta_{ij}=\eta_{ji}$), and $\mathcal N_i$ is the set of $i$'s neighbors. The factor of $\omega$ in the sum is needed to keep the size of the jump within $\varepsilon\kappa_{ij}(\theta_i)$. This is because $\theta_j(t)$ behaves like $\omega t$ when crosses zero and therefore the jump produced by $\delta(\theta_j(t))$ is of size $\int\delta(\theta_j(t))dt=\omega^{-1}$ \cite{Izhikevich99WeaklyPulseCoupled}.

The coupling function $\kappa_{ij}$ can be classified based on the qualitative effect it produces in the absence of delay. After one period, if the net effect of the mutual jumps brings a pair of oscillators closer,  we call it {\bf attractive} coupling. If the oscillators are brought further apart, it is considered to be {\bf repulsive} coupling. The former can be achieved for instance if $\kappa_{ij}(\theta)\leq0$ for $\theta\in[0,\pi)$ and $\kappa_{ij}(\theta)\geq0$ for $\theta\in[\pi,2\pi)$. See Figure \ref{fig:attractive_coupling} for an illustration of an attractive coupling $\kappa_{ij}$ and its effect on the relative phases.

%

\begin{figure}[htp]
\centering
\includegraphics[width=.8\columnwidth]{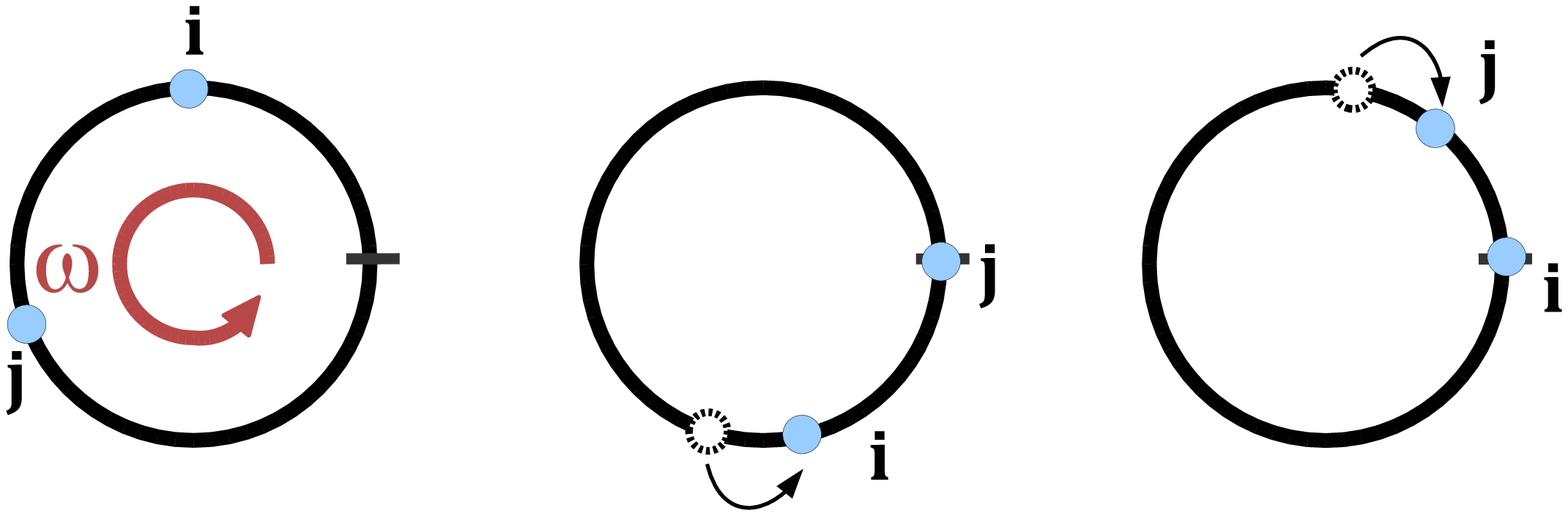}
\includegraphics[width=.78\columnwidth]{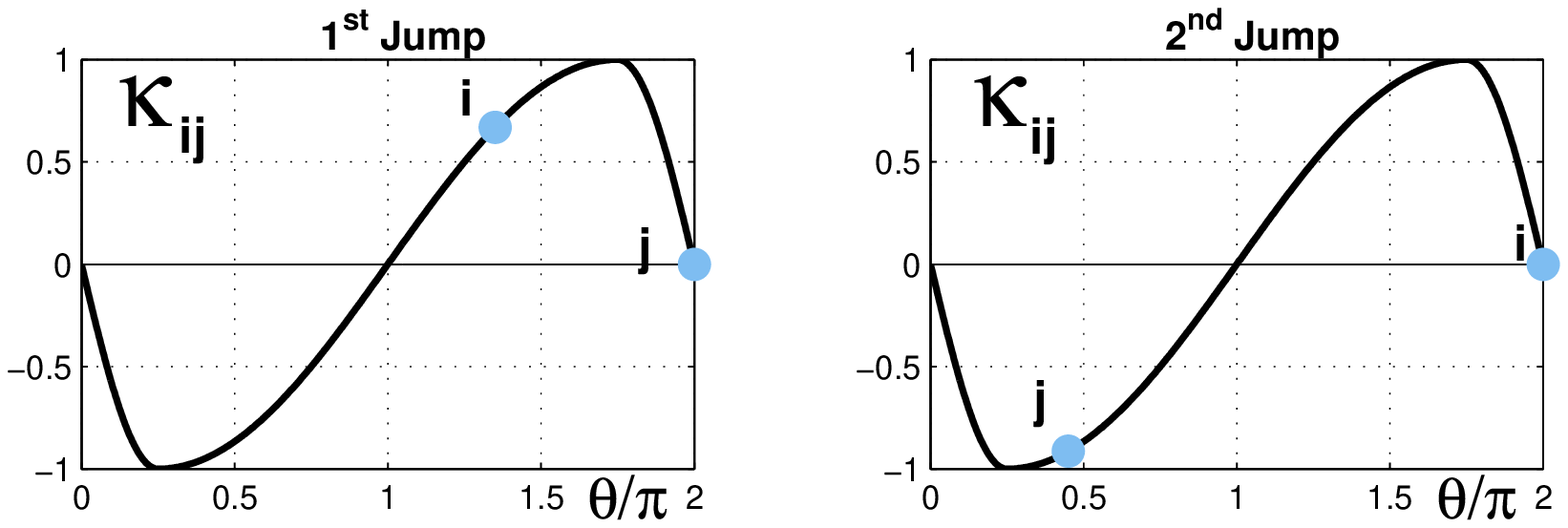}
\caption{{Pulse-coupled oscillators with attractive coupling.}}\label{fig:attractive_coupling}
\end{figure}


This pulse-like interaction between oscillators was first introduced by Peskin  \cite{Peskin75MathematicalAspects} in 1975 as a model of the pacemaker cells of the heart, although the canonic form did not appear in the literature until 1999 \cite{Izhikevich99WeaklyPulseCoupled}. In general, when the number of oscillators is large, there are several different limit cycles besides the in-phase synchronization and many of them can be stable  \cite{Ermentrout1992StablePeriodic}. 

The question of whether this system can collectively achieve in-phase synchronization  was answered for the complete graph case and zero delay by Mirollo and Strogatz in 1990 \cite{Mirollo&Strogatz90SynchronizationofPulseCoupled}. They showed if $\kappa_{ij}(\theta)$ is strictly increasing on $(0,2\pi)$ with a discontinuity in $0$ (which resembles attractive coupling), then for almost every initial condition, the system can synchronize in phase in the long run.



The two main assumptions of \cite{Mirollo&Strogatz90SynchronizationofPulseCoupled} are all to all comunication and zero delay. Whether in-phase synchronization can be achieved for arbitrary graphs has been an open problem for more than twenty years. On the other hand, when delay among oscillators is introduced the analysis becomes intractable. Even for the case of two oscillators, the number of possibilities to be considered is large \cite{Ernst&Pawelzik95synchronizationinduced,Ernst&Pawelzik98delayinduced}.

\subsection{Phase-coupled Oscillators}

In the model of phase-coupled oscillators, the interaction between neighboring oscillators $i$ and $j\in \mathcal N_i$ is modeled by change of the oscillating speeds. Although in general the speed change can be a function of both phases $(\theta_i,\theta_j)$, we concentrate on the case where the speed change is a function of the phase differences $f_{ij}(\phi_j(t-\eta_{ij}) - \phi_i(t))$. Thus, since the net speed change of oscillator $i$ amounts to the sum of the effects of its neighbors, the full dynamics is described by
\begin{equation}\label{eq:delayed_phase_model}
    \dot{\phi}_i(t)= \omega + \varepsilon\sum_{j\in \mathcal N_i}f_{ij}(\phi_j(t-\eta_{ij}) - \phi_i(t)).
\end{equation}
The function $f_{ij}$ is usually called coupling function, and as before $\eta_{ij}$ represents delay and $\mathcal N_i$ is the set of neighbors of $i$.

\begin{figure}[htp]
\centering
\includegraphics[width=.7\columnwidth]{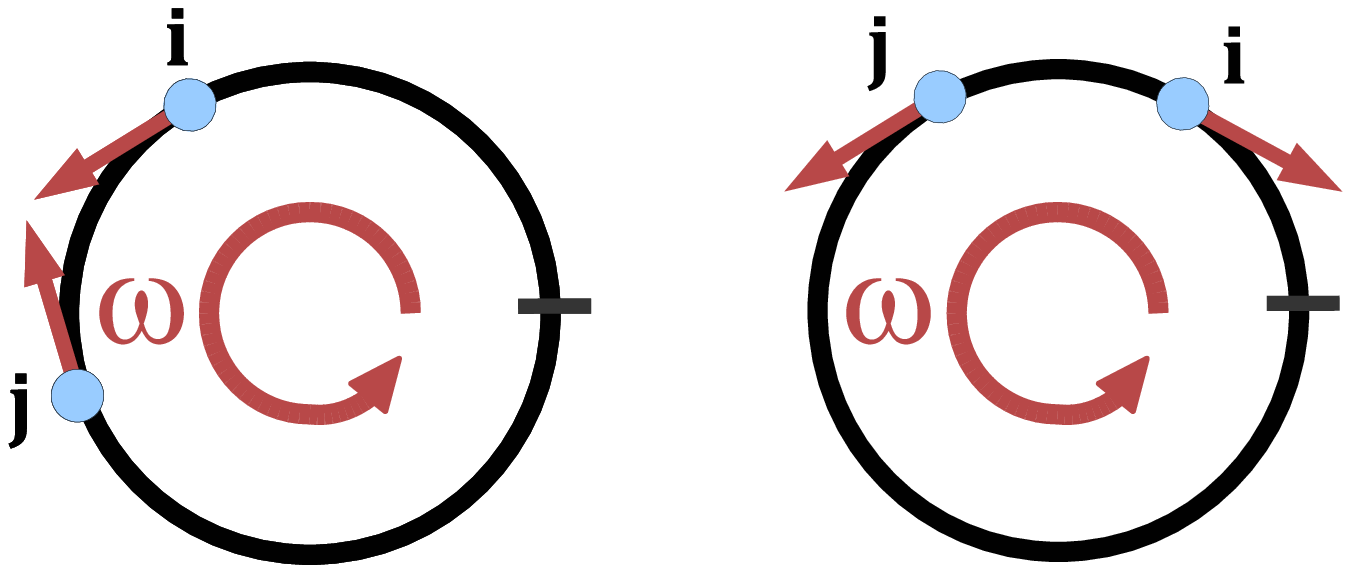}
\includegraphics[width=.78\columnwidth]{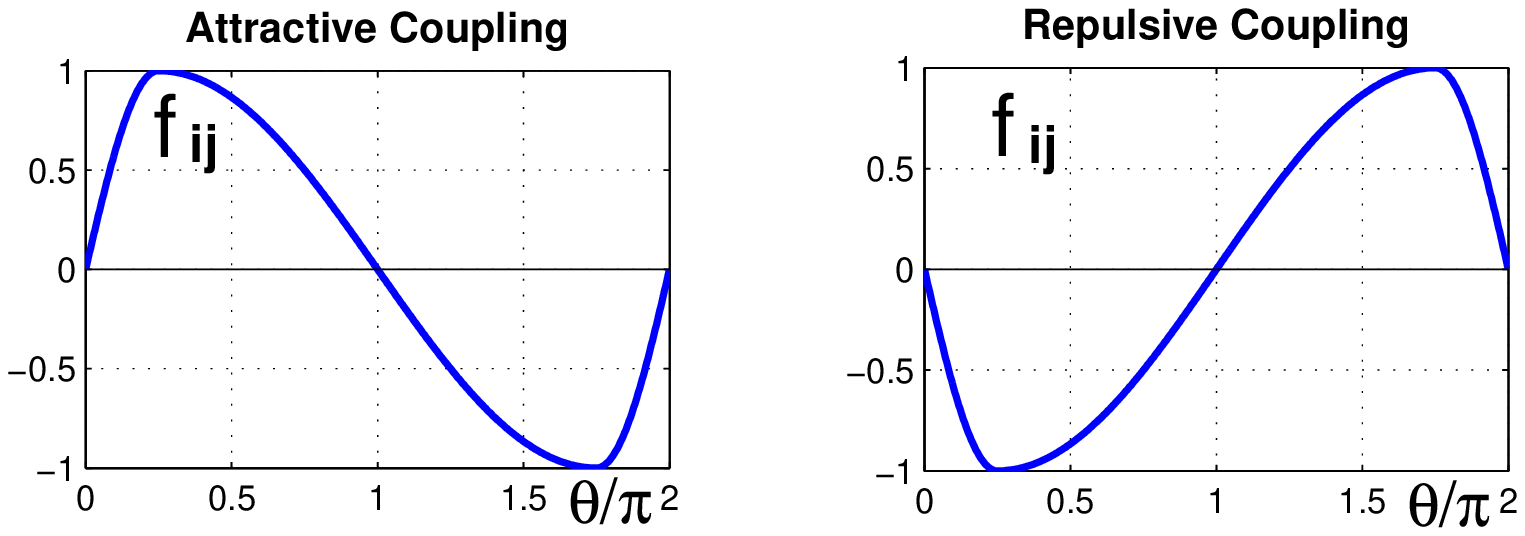}
\caption{{Phase-coupled oscillators with attractive and repulsive coupling. } } \label{fig:H_attractive_repulsive}
\end{figure}

A similar definition for attractive and repulsive couplings can be done in this model. We say that the coupling function $f_{ij}$ is {\bf attractive} if, without delays, the change in speeds brings oscillators closer, and {\bf repulsive} if they are brought apart. Figure \ref{fig:H_attractive_repulsive} shows typical attractive and repulsive coupling functions where arrows represent the speed change produced by the other oscillator; if the pointing direction is counter clockwise, the oscillator speeds up, and otherwise it slows down.

When $f_{ij}=\frac{K}{N}\sin()$, $K>0$ (attractive coupling), this model is known as the classical Kuramoto model \cite{Kuramoto84chemicaloscillations}. Intensive research work has been conducted on this model, however convergence results are usually limited to cases with all to all coupling ($\mathcal N_i=\mathcal N\backslash \{i\}$, i.e., complete graph topology) and no delay ($\eta_{ij}=0$), see e.g. \cite{Monzon&Paganini2005gobalconsidarations,Sepulchre2007},  or to some regions of the state space \cite{Papachristodoulou&Jadbabie2006SynchronizationinOscillator}.

\subsection{Weak Coupling Approximation}

We now concentrate in the regime in which the coupling strength of both models is weak, i.e. $1\gg\varepsilon>0$. For pulse-coupled oscillators, this implies that the effect of the jumps originated by each neighbor can be approximated by their average \cite{Izhikevich98PhaseModels}. For phase-coupled oscillators, it implies that to the first order $\phi_i(t-\eta_{ij})$ is well approximated by $\phi_i(t) - \omega\eta_{ij}$.

The effect of these approximations allows us to completely capture the behavior of both systems using the following equation where we assume that every oscillator has the same natural frequency $\omega$ and only keep track of the relative difference using
\begin{equation} \label{eq:phase_dynamics}
\dot{\phi}_i= \varepsilon\sum_{j\in \mathcal N_i}f_{ij}(\phi_j - \phi_i - \psi_{ij} ).
\end{equation}
For pulse-coupled oscillators, the coupling function is given by
\begin{equation}\label{eq:f_kappa}
f_{ij}(\theta)=\frac{\omega}{2\pi}\kappa_{ij}(-\theta),
\end{equation}
and the phase lag $\psi_{ij}=\omega \eta_{ij}$ represents the distance that oscillator $i$'s phase can travel along the unit circle during the delay time $\eta_{ij}$. Equation \eqref{eq:f_kappa} also shows that the attractive/repulsive coupling classification of both models are in fact equivalent, since in order to produce the same effect $\kappa_{ij}$ and $f_{ij}$ should be mirrored, as illustrated in Figure \ref{fig:attractive_coupling} and Figure \ref{fig:H_attractive_repulsive}.

Equation \eqref{eq:phase_dynamics} captures the relative change of the phases and therefore any solution to \eqref{eq:phase_dynamics} can be immediately translated to either \eqref{eq:pc_model} or \eqref{eq:delayed_phase_model} by adding $\omega t$. For example, if $\phi^*$ is an equilibrium of \eqref{eq:phase_dynamics}, by adding $\omega t$, we obtain a limit cycle in the previous models. { Besides the delay interpretation for $\psi_{ij}$, \eqref{eq:phase_dynamics} is also known as a system of coupled oscillators with {\it frustration}, see e.g. \cite{Daido1992QuasientrainmentandSlowRelaxation}}.

From now on we will concentrate on \eqref{eq:phase_dynamics} with the understanding that any convergence result derived will be immediately  true for the original models in the weak coupling limit. We are interested in the attracting properties of phase-locked invariant orbits within $\mathcal T^N$, which can be represented by
$
\phi(t)= \omega^*t\ones_N + \phi^*
$,
where  $\ones_N=(1,\dots,1)^T\in\mathcal T^N$, and $\phi^*$ and $\omega^*$ are solutions to
\begin{equation} \label{eq:fij_phase_locked_eqn}
\omega^*=\varepsilon\sum_{j\in \mathcal N_i} f_{ij}(\phi_j^* - \phi_i^* - \psi_{ij})
\mbox{,  $\forall i$.}
\end{equation}

Whenever the system reaches one of these orbits, we say that it is synchronized or phase-locked. If furthermore, all the elements of $\phi^*$ are equal, we say the system is synchronized {\bf in-phase} or  that it is {\bf in-phase} locked. It is easy to check that for a given equilibrium $\phi^*$ of \eqref{eq:phase_dynamics}, any solution of the form $\phi^*+\lambda\ones_N$, with $\lambda\in\mathds R$, is also an equilibrium that identifies the same limit cycle. Therefore, two equilibria $\phi^{1,*}$ and $\phi^{2,*}$ will be considered to be equivalent, if both identifies the same orbit, or equivalently, if both belongs to the same connected set of equilibria  
\begin{equation}
E_{\phi^*}:=\{\phi\in\mathcal T^N| \phi=\phi^*+\lambda\ones_N, \lambda\in \mathds R\}. \label{eq:E_phi}
\end{equation}

\section{Effect of Topology and Coupling}\label{sec:topology&coupling}

In this section we concentrate on the class of coupling function $f_{ij}$ that are symmetric ($f_{ij}=f_{ji}$ $\forall ij$), odd ( $f_{ij}(-\theta)=-f_{ij}(\theta)$) and continuously differentiable.
We also assume that there is no delay within the network, i.e. $\psi_{ij}=0$ $\forall ij$. Thus, \eqref{eq:phase_dynamics} reduces to
\begin{equation}
\label{eq:phase_dynamics_no_delay}
\dot{\phi}_i = \varepsilon\sum_{j\in \mathcal N_i}f_{ij}(\phi_j - \phi_i).
\end{equation}

In the rest of this section we progressively show how with some extra conditions on $f_{ij}$ we can guarantee in-phase synchronization for arbitrary undirected graphs. Since we know that the network can have many other phase-locked trajectories besides the in-phase one, our target is an {\bf almost global stability} result \cite{Rantzer2001dualtoLypunov}, meaning that the set of initial conditions that does not eventually lock in-phase has zero measure. Latter we show how most of the phase-locked solution that appear on a complete graph are unstable under some general conditions on the structure of the coupling function.

\subsection{Preliminaries}
\label{sec:prelim}
We now introduce some prerequisites used in our later analysis.
\subsubsection{ Algebraic Graph Theory } \label{ssec:graph_theory}
We start by reviewing basic definitions and properties from graph theory \cite{Bollobas98moderngraph,Godsil&Royle2001algebraicgraph} that are used in the paper.
Let $G$ be the connectivity graph that describes the coupling configuration. We $V(G)$ and $E(G)$ to denote the set of vertices ($i$ or $j$) and undirected edges ($e$) of $G$.   
An undirected graph $G$ can be directed by giving a specific orientation $\sigma$ to the elements in the set $E(G)$. That is, for any given edge $e\in E(G)$, we designate one of the vertices to be the {\it head} and the other to be the {\it tail} giving $G^\sigma$.

Although in the definitions that follow we need to give the graph $G$ a given orientation $\sigma$, the underlying connectivity graph of the system is assumed to be {\bf undirected}. This is not a problem as the properties used in this paper are independent of a particular orientation $\sigma$ and therefore are properties of the undirected graph $G$. Thus, to simplify notation we drop the superscript $\sigma$ from $G^\sigma$ with the understanding that $G$ is now an induced directed graph with some fixed, but arbitrarily chosen, orientation.

We use $P=(V^-,V^+)$ to denote a partition of the vertex set $V(G)$ such that $V(G)=V^- \cup V^+$ and $V^- \cap V^+=\emptyset$. The cut $C(P)$ associated with $P$, or equivalently $C(V^-,V^+)$, is defined as $C(P):=\{ij\in E(G)| i\in V^-, j\in V^+\text{, or vice versa.}\}$. Each partition can be associated with a vector column $c_P$ where $c_P(e)= 1$ if $e$ goes form $V^-$ to $V^+$, $c_P(e)=-1$ if $e$ goes form $V^+$ to $V^-$ and $c_P(e)= 0$ if $e$ stays within either set. 

There are several matrices associated with the oriented graph $G$ that embed information about its topology. However, the one with most significance to this work is the {\it oriented incidence matrix} $B \in \mathds R^{\norm{V(G)}\times \norm{E(G)}}$ where $B(i,e) =1$ if $i$ is the head of $e$, $B(i,e) =-1$ if $i$ is the tail of $e$ and $B(i,e) =0$ otherwise. 

\subsubsection{ Potential Dynamics }\label{ssec:potential}
We now describe how our assumptions on $f_{ij}$ not only simplifies considerably the dynamics, but also allows us to use the graph theory properties introduced in \ressec{graph_theory} to gain a deeper understanding of \eqref{eq:phase_dynamics}.

While $f_{ij}$ being continuously differentiable is standard in order to study local stability and sufficient to apply LaSalle's invariance principle \cite{Khalil96nonlinearsystems}, the symmetry and odd assumptions have a stronger effect on the dynamics. 

For example, under these assumptions the system \eqref{eq:phase_dynamics_no_delay} can be compactly
rewritten in a vector form as
\begin{equation} \label{eq:F_dynamics}
\dot{\phi}=-\varepsilon BF(B^T\phi)
\end{equation}
where $B$ is the adjacency matrix defined in \ressec{graph_theory} and
the map $F:\mathcal E(G) \rightarrow \mathcal E(G)$ is
\[
    F(y)=(f_{ij}(y_{ij}))_{{ij}\in E(G)}.
\]

This new representation has several properties. First, from the properties of $B$ one can easily show that \eqref{eq:fij_phase_locked_eqn} can only hold with $\omega^*=0$ for arbitrary graphs \cite{Brown&Holmes&Moehlis03globallycoupled} (since $N\omega^*=\omega^*\ones_N^T\ones_N=-\varepsilon\ones_N^TBF(B^T\phi)=0$), which implies that every phase-locked solution is an equilibrium of \eqref{eq:phase_dynamics_no_delay} and that every limit cycle of the original system  \eqref{eq:phase_dynamics} can be represented by some $E_\phi^*$ on \eqref{eq:phase_dynamics_no_delay}.

However, the most interesting consequence of \eqref{eq:F_dynamics} comes from interpreting  $F(y)$  as the gradient of a potential function
\[
    W(y)=\sum_{ij\in E(G)}\int_0^{y_{ij}}f_{ij}(s)ds.
\]
Then, by defining $V(\phi)= (W\circ B^T)(\phi)=W(B^T\phi)$, \eqref{eq:F_dynamics} becomes a gradient descent law for $V(\phi)$, i.e.,
\begin{align*}
\dot{\phi}=-\varepsilon BF(B^T\phi) =-\varepsilon B\nabla W(B^T\phi)= -\varepsilon \nabla V(\phi),
\end{align*}
where in the last step above we used the property $\nabla(W\circ B^T)(\phi)=B\nabla W(B^T\phi)$.
This makes $V(\phi)$ a natural Lyapunov function candidate since
\begin{equation}\label{eq:Vdot}
    \dot V(\phi) =\ip{\nabla V(\phi)}{\dot\phi}  =-\varepsilon \norm{\nabla V(\phi)}^2=-\frac{1}{\varepsilon}\norm{\dot\phi}^2 \leq 0.
\end{equation}

Furthermore, since the trajectories of \eqref{eq:F_dynamics} are constrained into the $N$-dimensional torus $\mathcal T^N$, which is compact, $V(\phi)$ satisfies the hipotesis of LaSalle's invariance principle (Theorem 4.4 \cite{Khalil96nonlinearsystems}), i.e. there is a compact positively invariant set, $\mathcal T^N$ and a function $V:\mathcal T^N\rightarrow \mathds R$ that decreases along the trajectories $\phi(t)$. Therefore, for every initial condition, the trajectory converges to the largest invariant set $M$ within $\{\dot V \equiv 0 \}$ which is the equilibria set $E=\{\phi\in\mathcal T^N| \dot\phi\equiv0\}=\bigcup_{\phi^*}E_{\phi^*}$.



\begin{remark}
The fact that symmetric and odd coupling induces potential dynamics is well know in the physics community~\cite{Hoppensteadt&Izhikevich97weaklyconnected}. However, it has been also rediscovered in the control community~\cite{Jadbabaie&Motee&Barahona2004stabilityofkuramoto} for the specific case of sine coupling. Clearly, this is not enough to show almost global stability, since it is possible to have other stable phase-locked equilibrium sets besides the in-phase one. However, if we are able show that all the non-in-phase equilibria are unstable, then almost global stability follows. That is the focus of the next section.
\end{remark}

\subsection{Negative Cut Instability Condition}
\label{sec:main}

We now present the main results of this section. Our technique can be viewed as a generalization of \cite{Monzon&Paganini2005gobalconsidarations}. 
By means of algebraic graph theory, we provide a better stability analysis of the equilibria under a more general framework. We further use the new stability results to characterize $f_{ij}$ that guarantees almost global stability.

\subsubsection{Local Stability Analysis} \label{ssec:local_stability}

In this section we develop the graph theory based tools to characterize the stability of each equilibrium. We will show that given an equilibrium $\phi^*$ of the system \eqref{eq:F_dynamics}, with connectivity graph $G$ and $f_{ij}$ as described in this section. If there exists a cut $C(P)$ such that the sum
\begin{equation} \label{eq:cut_instability_cond}
    \sum_{ij\in C(P)} f'_{ij}(\phi_j^* - \phi_i^*) < 0,
\end{equation}
the equilibrium $\phi^*$ is {\bf unstable}.

Consider first an equilibrium point $\phi^*$. Then, the first order approximation of \eqref{eq:F_dynamics} around $\phi^*$ is
\begin{align*}
\delta\dot\phi
&=-\varepsilon B\left[\frac{\partial }{\partial y}F(B^T\phi^*)\right]B^T\delta\phi,
\end{align*}
were $\delta\phi=\phi - \phi^*$ is the incremental phase variable, and $\frac{\partial }{\partial y}F(B^T\phi^*) \in \mathds R^{\norm{ E(G)}\times \norm{ E(G)}}$is the Jacobian of $F(y)$ evaluated at $B^T\phi^*$, i.e., 
$
\frac{\partial }{\partial y}F(B^T\phi^*)=
\diag\left({\{f'_{ij}(\phi^*_j-\phi^*_i)\}_{{ij}\in E(G)}}\right).
$

Now let $A=-\varepsilon B\left[\frac{\partial }{\partial y}F(B^T\phi^*)
\right]B^T$ and consider the linear system
$
\delta\dot\phi=A\delta\phi.
$
Although it is possible to numerically calculate the eigenvalues of $A$ given $\phi^*$ to study the stability, here we use the special structure of $A$
to provide a sufficient condition for instability that has nice graph theoretical interpretations.

Since $A$ is symmetric, it is straight forward to check that $A$ has at least one positive eigenvalue, i.e. $\phi^*$ is unstable, if and only if $x^TAx>0$. 
Now, given any partition $P=(V^-,V^+)$, consider the associated vector $c_P$, define $x_P$ such that $x_i=\frac{1}{2}$ if $i\in V^+$ and $x_i=-\frac{1}{2}$ if $i\in V^-$. Then it follow from the definition of $B$ that $c_P=B^Tx_P$ which implies that
\begin{align*}
\frac{-1}{\varepsilon}x_P^TAx_P = c_P^T\left[\frac{\partial }{\partial y}F(B^T\phi^*)\right]c_P =\sum_{ij\in C(P)} f'_{ij}(\phi_j^*-\phi_i^*).
\end{align*}

Therefore, when condition \eqref{eq:cut_instability_cond} holds,  $A=-\varepsilon BDB^T$ has at least one eigenvalue whose real part is {\bf positive}.

\begin{remark}
Equation \eqref{eq:cut_instability_cond} provides a {\bf sufficient} condition for instability; it is not clear what happens when \eqref{eq:cut_instability_cond} does not hold. However, it gives a graph-theoretical interpretation  that can be used to provide stability results for general topologies. That is, if the {\bf minimum cut cost } is negative, the equilibrium is unstable.
\end{remark}

\begin{remark}\label{rem:generalization}
Since the weights of the graph $f'_{ij}(\phi_j^*-\phi_i^*)$ are functions of the phase difference, \eqref{eq:cut_instability_cond} holds for any equilibria of the form $\phi^*+\lambda \ones_N$. Thus, the result holds for the whole set $E_{\phi^*}$ defined in \eqref{eq:E_phi}.
\end{remark}

When \eqref{eq:cut_instability_cond} is specialized to $P=(\{i\},V(G)\backslash \{i\})$ and $f_{ij}(\theta)=\sin(\theta)$,  it reduces to the instability condition in Lemma 2.3 of \cite{Monzon&Paganini2005gobalconsidarations}; i.e.,
\begin{equation}\label{eq:neighbors_instability_condition}
    \sum_{j\in\mathcal N_i}cos(\phi^*_j- \phi^*_i)<0.
\end{equation}
However, \eqref{eq:cut_instability_cond} has a broader applicability spectrum as the following example shows.

\begin{example}

Consider a six oscillators network as in \refi{network_example}, where each node is linked with its four closest neighbors and $f_{ij}(\theta)=\sin(\theta)$. Then, by symmetry, it is easy to verify that
\begin{equation}\label{eq:equilibrium}
    \phi^*=\left[0, \frac{\pi}{3},\frac{2\pi}{3},\pi,\frac{4\pi}{3},
    \frac{5\pi}{3}\right]^T
\end{equation}
is an equilibrium of \eqref{eq:phase_dynamics_no_delay}.

\begin{figure}[htp]
\centering
\includegraphics[width=.4\columnwidth]{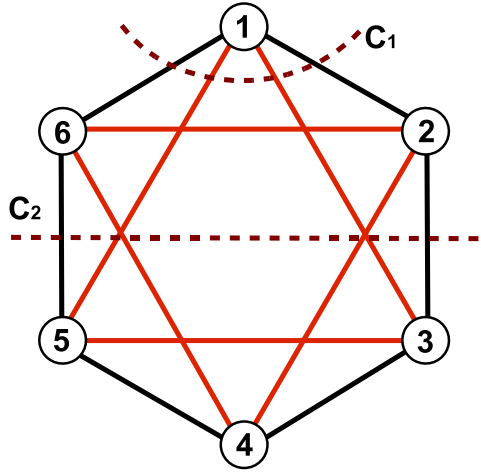}
\caption{The network of six oscillators (Example 4)}\label{fig:network_example}
\end{figure}

\begin{figure}[htp]
\centering
\includegraphics[width=.95\columnwidth,height=.4\columnwidth]{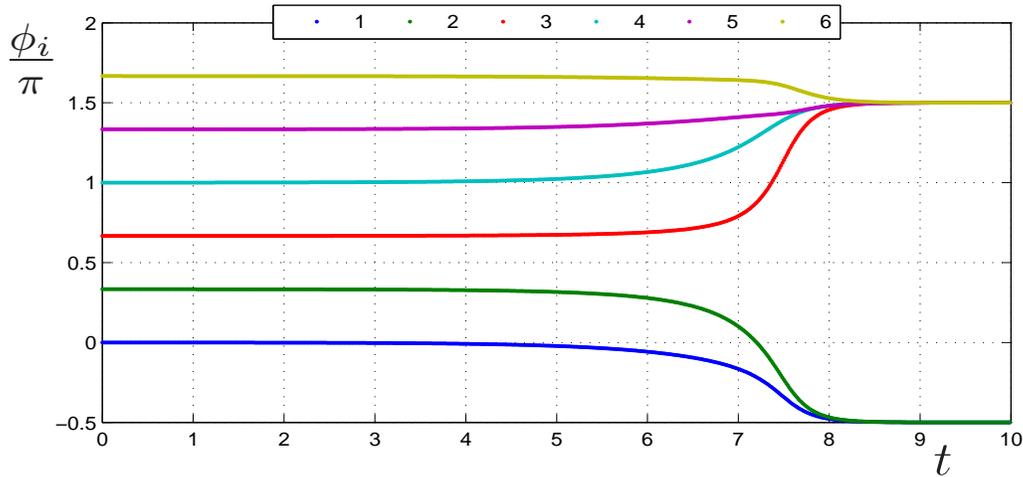}
\caption{Unstable equilibrium $\phi^*$. Initial condition $\phi_0=\phi^* +\delta\phi$}\label{fig:phases_example}
\end{figure}

We first study the stability of $\phi^*$ using \eqref{eq:neighbors_instability_condition}  as in \cite{Monzon&Paganini2005gobalconsidarations}. By substituting \eqref{eq:equilibrium} in $\cos(\phi^*_j-\phi^*_i)$ $\forall ij\in E(G)$ we find that the edge weights can only take two values:
\begin{equation*}
\cos(\phi^*_j-\phi^*_i)=
\begin{cases}
\cos(\frac{\pi}{3}) =\frac{1}{2},  & \text{ if $j = i \pm 1 \mod 6$ } \\
\cos(\frac{2\pi}{3}) =-\frac{1}{2}, & \text{ if $j = i \pm 2 \mod 6$ }
\end{cases}
\end{equation*}
Then, since any cut that isolates one node from the rest (like $C_1=C(\{1\},V(G)\backslash\{1\})$ in  \refi{network_example}) will always have two edges of each type, their sum is {\bf zero}. Therefore, \eqref{eq:neighbors_instability_condition} cannot be used to determine stability.

If we now use condition \eqref{eq:cut_instability_cond} instead, we are allowed to explore a wider variety of cuts that can potentially have smaller costs. In fact, if instead of $C_1$ we sum over $C_2=C(\{1,2,6\},\{3,4,5\})$, we obtain,
\[
\sum_{ij\in C_2}\cos(\phi^*_j- \phi^*_i)= -1<0,
\]
which implies that $\phi^*$ is unstable.

Figure \ref{fig:phases_example} verifies the equilibrium instability. By starting with an initial condition $\phi_0=\phi^* +\delta\phi$ close to the equilibrium $\phi^*$, we can see how the system slowly starts to move away from $\phi^*$ towards a {\bf stable} equilibrium set.

Furthermore, we can study the whole family of non-isolated equilibria given by 
\begin{equation}\label{eq:eq_family}
    \phi^*=\left[\varepsilon_1, \frac{\pi}{3} +\varepsilon_2,\frac{2\pi}{3} +\varepsilon_3,\pi+\varepsilon_1,\frac{4\pi}{3} +\varepsilon_2,
    \frac{5\pi}{3} +\varepsilon_3\right]^T
\end{equation}
where $\varepsilon_1,\varepsilon_2,\varepsilon_3\in\mathds R$, which due to Remark \ref{rem:generalization}, we can reduce \eqref{eq:eq_family} to  
\begin{equation}\label{eq:eq_family2}
    \phi^*=\left[0, \frac{\pi}{3} +\lambda_1,\frac{2\pi}{3} +\lambda_2,\pi,\frac{4\pi}{3} +\lambda_1,
    \frac{5\pi}{3} +\lambda_2\right]^T
\end{equation}
with $\lambda_1=\varepsilon_2-\varepsilon_1$ and $\lambda_2=\varepsilon_3-\varepsilon_1$. 

\begin{figure}[htp]

\includegraphics[height=.5\columnwidth,width=.95\columnwidth]
{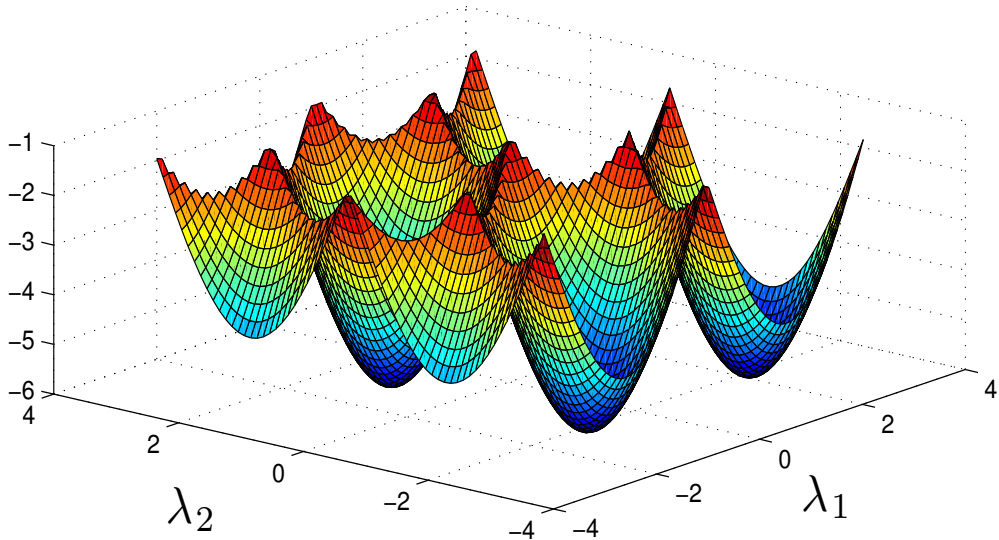}
\caption{Minimum cut value $C^*(\lambda_1,\lambda_2)$ showing that the equilibria \eqref{eq:eq_family} are unstable}\label{fig:eq_family}
\end{figure}

Instead of focusing on only one cut, here we compute the minimum cut value \eqref{eq:cut_instability_cond} over the 31 possible cuts, i.e. $C^*(\lambda_1,\lambda_2):=\min_P \sum_{ij\in C(P)} f'_{ij}(\phi_j(\lambda_1,\lambda_2)^*-\phi_i^*(\lambda_1,\lambda_2))$.
 Figure \ref{fig:eq_family} show the value of $C^*(\lambda_1,\lambda_2)$ for $\lambda_i\in[-\pi,\pi]$. Since $C^*(\lambda_1,\lambda_2)$ is $2\pi$-periodic on each variable and its value is negative for every $\lambda_1,\lambda_2\in[-\pi,\pi]$, the family of equilibria \eqref{eq:eq_family2} (and consequently \eqref{eq:eq_family}) is unstable.

\end{example}

\subsubsection{Almost Global Stability} \label{ssec:global_stability}

Condition \eqref{eq:cut_instability_cond} also  provides insight on which class of coupling functions can potentially give us almost global convergence to the in-phase equilibrium set $E_{\ones_N}$. If it is possible to find some $f_{ij}$ with $f'_{ij}(0)>0$, such that for any non-in-phase equilibrium $\phi^*$, there is a cut $C$ with $\sum_{ij\in C}{f'_{ij}(\phi_j^* - \phi_i^*)}<0$, then the in-phase equilibrium set will be almost globally stable \cite{Ermentrout1992StablePeriodic}. The main difficulty is that for general $f_{ij}$ and arbitrary network $G$, it is not easy to locate every phase-locked equilibria and thus, it is not simple to know in what region of the domain of $f_{ij}$ the slope should be negative.

We now concentrate on the one-parameter family of functions $\mathcal F_b$, with $b\in(0,\pi)$, such that $f_{ij}\in\mathcal F_b$ whenever $f_{ij}$ is {\bf symmetric}, {\bf odd}, {\bf continuously differentiable} and 
\begin{itemize}
\item $f'_{ij}(\theta;b)>0 \text{, } \forall \theta \in (0,b)\cup(2\pi-b,2\pi)$, and
\item $f'_{ij}(\theta;b)<0 \text{, } \forall \theta \in (b,2\pi-b)$.
\end{itemize}
See Figure \ref{fig:H_attractive_repulsive} for an illustration with $b=\frac{\pi}{4}$. Also note that this definition implies that if $f_{ij}(\theta;b)\in\mathcal F_b$, the coupling is attractive and $f_{ij}(\theta;b)>0$ $\forall \theta \in (0,\pi)$. This last property will be used later. We also assume the graph $G$ to be {\bf connected}.

%


In order to obtain almost global stability we need $b$ to be small. However, since the equilibria position is not known a priori, it is not clear how small $b$ should be or if there is any $b>0$ such that all nontrivial equilibria are unstable. We therefore first need to estimate the region of the state space that contains every non-trivial phase-locked solution.

Let $I$ be a compact connected subset of $\mathds S^1$ and let $l(I)$ be its length, e.g., if $I=\mathds S^1$ then $l(I)=2\pi$.
For any $S\subset V(G)$ and $\phi \in\mathcal T^N$,  define $d(\phi,S)$ as the length of the smallest interval $I$ such that $\phi_i\in I$ $\forall i\in S$, i.e.
\[
d(\phi,S)=l(I^*)=\min_{I:\phi_i\in I\text{, }\forall i\in S} l(I).
\]


Using this metric, together with the aid of Proposition 2.6 of
\cite{Brown&Holmes&Moehlis03globallycoupled} we can identify two very insightful properties of the family $\mathcal F_b$ whenever the graph $G$ is connected.

\begin{claim}\label{lemma:non_existence_of_non_trivial_eq}
If $\phi^*$ is an equilibrium point of \eqref{eq:F_dynamics} with $d(\phi^*,V(G)) \leq\pi$, then either $\phi^*$ is an in-phase equilibrium, i.e. $\phi^*=\lambda\ones_N$ for $\lambda\in \mathds R$, or has a cut $C$ with $f'_{ij}(\phi^*_j-\phi^*_i)<0$ $\forall ij\in C$.
\end{claim}
\begin{proof}
Since $d(\phi^*,V(G)) \leq\pi$, all the phases are contained in a half circle and for the oscillator with smallest phase $i_0$, all the phase differences $(\phi^*_j-\phi^*_{i_0})\in[0,\pi]$. However, since $f_{ij}(\cdot;b)\in \mathcal F_b$ implies  $f_{ij}(\theta;b)\geq0$ $\forall \theta \in [0,\pi]$ with equality only for $\theta\in\{0,\pi\}$, $\dot \phi^*_{i_0}=\sum_{j\in\mathcal N_{i_0}} f_{ij}(\phi^*_j-\phi^*_{i_0})=0$ can only hold if $\phi_j^*-\phi_{i_0}^*\in\{0,\pi\}$ $\forall j\in\mathcal N_i$.  Now let $V^-=\{ i\in V(G): d(\phi^*,\{i,i_0\})=0\}$ and $V^+=V(G)\backslash V^-$. If $V^-=V(G)$, then $\phi^*$ is an in-phase equilibrium. Otherwise, $\forall ij\in C(V^-,V^+)$,  $f_{ij}(\phi_j^*-\phi_i^*)=f_{ij}(\pi)<0$. 
\end{proof}

We are now ready to establish a bound on the value of $b$ that guarantees the instability of the non-in-phase equilibria.

\begin{claim}\label{lemma:bound_on_b}
Consider $f_{ij}(\cdot;b)\in \mathcal F_b$ $\forall ij \in E(G)$ and arbitrary connected graph $G$. Then for any $b\leq\frac{\pi}{N-1}$ and non-in-phase equilibrium $\phi^*$, there is a cut $C$ with
$
f'_{ij}(\phi_j^*-\phi_i^*;b)<0, \forall ij\in C
$
\end{claim}
\begin{proof}
Suppose there is a non-in-phase equilibrium $\phi^*$ for which no such cut $C$ exists. Let $V^-_0=\{i_0\}$ and
$V^+_0=V(G)\backslash \{i_0\}$ be a partition of $V(G)$ for some arbitrary node $i_0$.

Since such $C$ does not exists, there exists some edge $i_0j_1\in C(V^-_0,V^+_0)$,
with $j_1\in V^+_0$, such that $f'_{i_0j_1}(\phi_{j_1}^*-\phi_{i_0}^*;b)\geq0$. Move $j_1$ from one side to the other of the partition by defining $V^-_1:=V^-_0 \cup \{j_1\}$ and
$V^+_1:=V^+_0 \backslash \{j_1\}$. Now since $f'_{i_0j_1}(\phi_{j_1}^*-\phi_{i_0}^*;b)\geq0$,
then
\[
d(\phi^*,V^-_1)\leq b.
\]
In other words, both phases should be within a distance smaller than $b$.

Now repeat the argument $k$ times. At the $k^{th}$ iteration, given $V^-_{k-1}$, $V^+_{k-1}$, again
we can find some $i_{k-1}\in V^-_{k-1}$, $j_k\in V^+_{k-1}$ such that $i_{k-1}j_k \in  C(V^-_{k-1},V^+_{k-1})$ and
$f'_{i_{k-1}j_k}(\phi_{j_k}^*-\phi_{i_{k-1}}^*;b)\geq0$.
Also, since at each step $d(\phi^*, \{i_{k-1},j_k\})\leq b$,
\[
d(\phi^*,V^-_k)\leq b + d(\phi^*,V^-_{k-1}).
\]
Thus by solving the recursion we get:
$
d(\phi^*,V^-_k)\leq kb.
$

After $N-1$ iterations we have $V^-_{N-1}=V(G)$ and $d(\phi^*,V(G))\leq (N-1)b$.
Therefore, since $b\leq\frac{\pi}{N-1}$, we obtain
\[
d(\phi^*,V(G))\leq(N-1)\frac{\pi}{N-1}=\pi.
\]
Then, by Claim \ref{lemma:non_existence_of_non_trivial_eq} $\phi^*$ is either an in-phase equilibrium or there is a cut $C$ with $f'_{ij}(\phi^*_j-\phi^*_i)<0$ $\forall ij\in C$. Either case gives a contradiction to assuming that $\phi^*$ is a non-in-phase equilibrium and $C$ does not exists. Therefore, for any non-in-phase $\phi^*$ and $b\leq\frac{\pi}{N-1}$, we can
always find a cut $C$ with $f_{ij}(\phi_j^*-\phi_i^*;b)<0$, $\forall ij\in C$.
\end{proof}

Claim \ref{lemma:bound_on_b} allows us to use our cut condition \eqref{eq:cut_instability_cond} on every non-in-phase equilibrium.
Thus, since \eqref{eq:F_dynamics} is a potential dynamics (c.f. Section \ref{ssec:potential}), from every initial condition the system converges to the set of equilibria $E$. But when $b\leq\frac{\pi}{N-1}$ the only stable equilibrium set inside $E$ is the in-phase set $E_{\ones_N}$. Thus, $E_{\ones_N}$ set is globally asymptotically stable.  We summarized this result in the following Proposition.

%
%
\begin{proposition}[Almost global stability]
Consider $f_{ij}(\theta;b)\in \mathcal F_b$ and an arbitrary connected
graph $G$. Then, if $b\leq\frac{\pi}{N-1}$, the in-phase equilibrium set $E_{\ones_N}$ is {\bf almost
globally asymptotically stable}.
\label{prop:main}
\end{proposition}


This result provides a sufficient condition for almost global asymptotic stability to the in-phase equilibrium set $E_{\ones_N}$. {Although found independently, the same condition was proposed} for a specific piecewise linear $f_{ij}$ in \cite{Sarlette2009PHDThesis}. Here we extend \cite{Sarlette2009PHDThesis} in many aspects.  For example,  instead of assuming equal coupling for every edge, our condition describes a large family of coupling functions $\mathcal F_b$ where each $f_{ij}$ can be taken independently from $\mathcal F_b$. Also, in~\cite{Sarlette2009PHDThesis} the construction of $f_{ij}(\theta)$ assumes a discontinuity on the derivative at $\theta=b$. This can pose a problem if the equilibrium $\phi^*$ happens to have phase differences $\phi_j^*-\phi^*_i=b$. Here we do not have such problem as $f_{ij}$ is continuously differentiable. 

The condition $b\leq\frac{\pi}{N-1}$ implies that, when $N$ is large, $f_{ij}$ should be decreasing in most of it domain. Using \eqref{eq:f_kappa} this implies that $\kappa_{ij}$ should be increasing within the region $(b,2\pi-b)$, which is similar to the condition on \cite{Mirollo&Strogatz90SynchronizationofPulseCoupled} and equivalent when $b\rightarrow 0$. Thus, Proposition \ref{prop:main} confirms the conjecture of~\cite{Mirollo&Strogatz90SynchronizationofPulseCoupled} by extending their result to arbitrary topologies and a more realistic continuous $\kappa_{ij}$ for the system \eqref{eq:pc_model} in the weak coupling limit. 


\subsection{Complete Graph Topology with a Class of Coupling Functions}
\label{sec:coupling}
In this subsection we investigate how conservative the value of $b$ found in Section \ref{ssec:global_stability} is for the complete graph topology. We are motivated by the results of \cite{Monzon&Paganini2005gobalconsidarations} where it is shown that $f(\theta)=\sin(\theta)$ ($b=\frac{\pi}{2}$) with complete graph topology ensures almost global synchronization.

Since for general $f$ it is not easy to characterize all the possible equilibria of the system, we study the stability of the equilibria that appear due to the equivalence of \eqref{eq:F_dynamics} with respect to the action group $S_N\times T^1$, where $S_N$ is the group of permutations of the $N$ coordinates and $T^1=[0,2\pi)$ represents the group action  of phase shift of all the coordinates, i.e. the action of $\delta\in T^1$ is $\phi_i\mapsto\phi_i+\delta$ $\forall i$.
We refer the readers to \cite{Ashwin&Swift92thedynamics} and \cite{Brown&Holmes&Moehlis03globallycoupled} for a detailed study of the effect of this property.

These equilibria are characterized by the isotropy subgroups $\Gamma$ of $S_N\times T^1$ that keep them fixed, i.e., $\gamma\phi^*=\phi^*$ $\forall \gamma\in\Gamma$. In \cite{Ashwin&Swift92thedynamics} it was shown that this isotropy subgroup takes the form of
\[
(S_{k_0}\times S_{k_1}\times\dots\times S_{k_{l_B-1}})^m\rtimes Z_m
\]
where $k_i$ and $m$ are positive integers such that $(k_0+k_1+\dots+k_{l_B-1})m=N$, $S_j$ is the permutation subgroup of $S_N$ of $j$-many coordinates and $Z_m$ is the cyclic group with action $\phi_i\mapsto \phi_i+\frac{2\pi}{m}$. The semiproduct $\rtimes$ represents the fact that $Z_m$ does not commute with the other subgroups.

In other words, each equilibria with isotropy $(S_{k_0}\times S_{k_1}\times\dots\times S_{k_{l_B-1}})^m\rtimes Z_m$ is conformed by $l_B$ shifted constellations $C_l$ ($l\in \{0,1,\dots l_B-1\}$) of $m$ evenly distributed blocks, with $k_l$ oscillators per block. We use $\delta_l$ to denote the phase shift between constellation $C_0$ and $C_l$. See Figure \ref{fig:symmetric_equilibria} for examples these types of equilibria.

\begin{figure}[htp]
\centering
\includegraphics[width=.3\columnwidth]{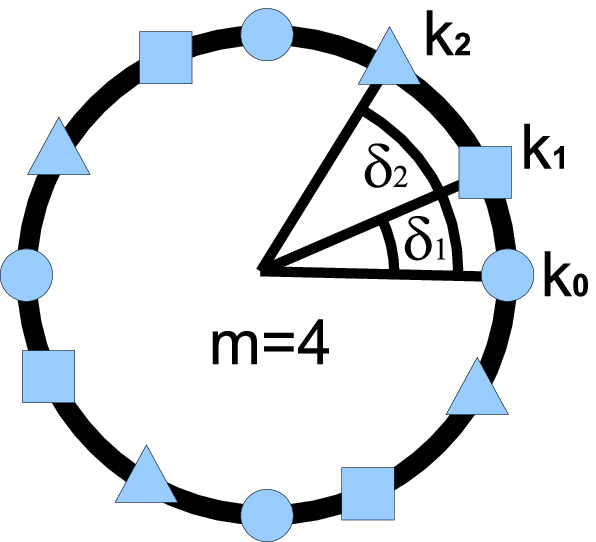}
\includegraphics[width=.3\columnwidth]{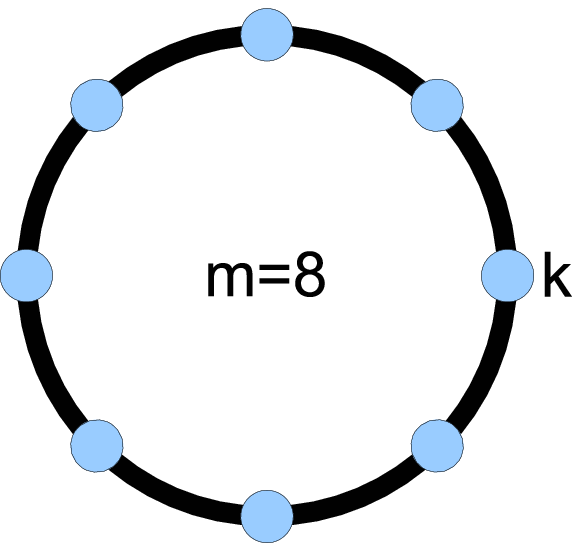}
\caption{ Equilibria with isotropy $(S_{k_0}\times S_{k_1} \times S_{k_2})^4\rtimes Z_4$ (left) and $(S_k)^8\rtimes Z_8$ (right)}\label{fig:symmetric_equilibria}
\end{figure}

Here we will show that under mild assumptions on $f$ and for $b=\frac{\pi}{2}$ most of the equilibria found with these characteristics are unstable. We first study all the equilibria with $m$ even. In this case there is a special property that can be exploited. 

That is, when $f\in \mathcal F_{\frac{\pi}{2}}$ such that $f$ is even around  $\frac{\pi}{2}$, we have
\begin{align}
g_m(\delta)&:=\sum_{j=0}^{m-1}f(\frac{2\pi}{m}j + \delta) \label{eq:g_m}
\\&=\sum_{j=0}^{{m}/{2}-1}f(\frac{2\pi}{m}j + \delta)+f(\pi+\frac{2\pi}{m}j + \delta) \nonumber
\\&=\sum_{j=0}^{{m}/{2}-1}f(\frac{2\pi}{m}j + \delta)+f((\frac{3\pi}{2}+\frac{2\pi}{m}j + \delta)-\frac{\pi}{2})\nonumber \\
&=\sum_{j=0}^{{m}/{2}-1}f(\frac{2\pi}{m}j + \delta)+f(-(\frac{2\pi}{m}j + \delta))\nonumber
\\&=\sum_{j=0}^{{m}/{2}-1}f(\frac{2\pi}{m}j + \delta)-f(\frac{2\pi}{m}j + \delta)=0 \nonumber
\end{align}
where the third step comes from $f$ being even around $\pi/2$ and $2\pi$-periodic, and the fourth from $f$ being odd.


Having $g_m(\delta)=0$ is the key to prove the instability of every equilibria with even $m$. It essentially states that the aggregate effect of one constellation $C_l$ on any oscillator $j\in V(G)\backslash C_l$ is zero when $m$ is even, and therefore any perturbation that maintains $C_l$  has null effect on $j$. This is shown in the next proposition.


\begin{figure}[htp]
\centering
\includegraphics[width=.4\columnwidth]{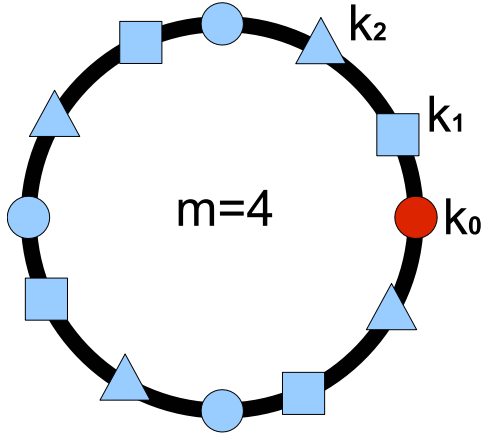}
\caption{ Cut of Proposition \ref{prop:m_even}, the red block represents one possible set $V_0$ }\label{fig:thm_m_even}
\end{figure}

\begin{proposition}[Instability for even $m$ ]\label{prop:m_even}
Given an equilibrium $\phi^*$ with isotropy  $(S_{k_1}\times S_{k_2}\times\dots\times S_{k_{l_B}})^m\rtimes Z_m$ and $f\in \mathcal F_{\frac{\pi}{2}}$ even around $\frac{\pi}{2}$. Then, if $m$ is even, $\phi^*$ is unstable.
\end{proposition}

\begin{proof}
We will show the instability of $\phi^*$ by finding a cut of the network satisfying \eqref{eq:cut_instability_cond}. Let $V_0\subset V(G)$ be the set of nodes within one of the blocks of the constellation $C_0$ and consider the partition induced by $V_0$, i.e. $P=(V_0,V(G)\backslash V_0)$. Due to the structure of $\phi^*$, \eqref{eq:cut_instability_cond} becomes
\begin{align*}
\sum_{ij\in C(P)} f'(\phi_j^* - \phi_i^*)&=-k_1f'(0)+\sum_{l=1}^{l_B}k_l g_m'(\delta_l),
\end{align*}
where $g_m'(\delta)$ is the derivative of $g_m$ and $\delta_l$ is the phase shift between the $C_0$ and $C_l$.
Finally, since by assumptions $g_m(\delta)\equiv 0$ $\forall \delta$ then it follows that $g_m'(\delta)\equiv 0$ and
\begin{align*}
\sum_{ij\in C(P)} f'_{ij}(\phi_j^* - \phi_i^*)&=-k_1f'(0)<0.
\end{align*}
Therefore, by  \eqref{eq:cut_instability_cond}, $\phi^*$ is unstable.
\end{proof}

The natural question that arises is whether similar results can be obtained for $m$ odd. The main difficulty in this case is that $g_m(\delta)=0$ does not  hold since we no longer evaluate $f$ at points with phase difference equal to $\pi$ such that they cancel each other. Therefore, an extra monotonicity condition needs to be added in order to partially answer this question. These conditions and their effects are summarized in the following claims.

\begin{claim}[Monotonicity]\label{claim:monotonicity}
Given $f\in \mathcal F_{\frac{\pi}{2}}$ such that $f$ is strictly concave for $\theta\in[0,\pi]$, then
\begin{align}
f'(\theta)-f'(\theta-\phi)<0,\quad 0\leq\theta-\phi<\theta\leq\pi \label{eq:mon_decreasing}\\
f'(\theta)-f'(\theta+\phi)<0,\quad -\pi\leq\theta<\theta+\phi\leq0 \label{eq:mon_increasing}
\end{align}
\end{claim}
\begin{proof} The proof is a direct consequence of the strict concavity of $f$. Since $f(\theta)$ is strictly concave then basic convex analysis shows that $f'(\theta)$ is strictly decreasing  within $[0,\pi]$. Therefore, the inequality \eqref{eq:mon_decreasing} follows directly from the fact that $\theta\in[0,\pi]$,$\theta-\phi\in[0,\pi]$ and $\theta-\phi<\theta$. To show \eqref{eq:mon_increasing} it is enough to notice that since $f$ is odd ( $f\in\mathcal F_{\frac{\pi}{2}}$), $f$ is strictly convex in $[\pi,2\pi]$. The rest of the proof is analogous to \eqref{eq:mon_decreasing}.
\end{proof}

\begin{claim}[$f'$ Concavity]\label{claim:concavity}
Given $f\in \mathcal F_{\frac{\pi}{2}}$ such that $f'$ is strictly concave for $\theta\in[-\frac{\pi}{2},\frac{\pi}{2}]$. Then for all $m\geq 4$, $f'(\frac{\pi}{m})\geq \frac{1}{2}f'(0)$.
\end{claim}
\begin{proof}
Since $f'(\theta)$ is concave for $\theta\in[-\pi,\pi]$ then it follows
\begin{align*}
f'(\frac{\pi}{m})&=f'(\lambda_m 0+(1-\lambda_m) \frac{\pi}{2})> \lambda_m f'(0) + (1-\lambda_m)f'(\frac{\pi}{2})> \lambda_m f'(0)
\end{align*}
where $\lambda_m=\frac{m-2}{m}$. Thus, for $m\geq4$, $\lambda_m\geq\frac{1}{2}$ and
\[
f'(\frac{\pi}{m})> \frac{1}{2}f'(0)
\]
as desired.
\end{proof}


\begin{figure}[htp]
\centering
\includegraphics[width=.4\columnwidth]{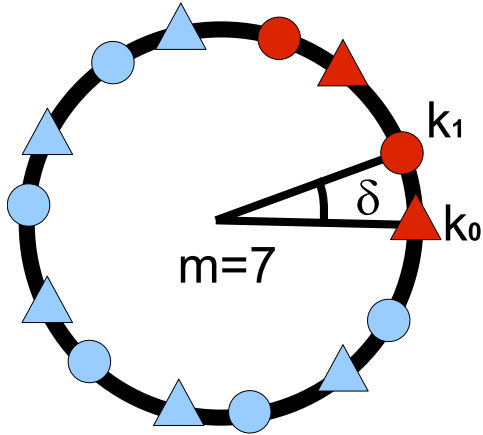}
\caption{ Cut used in Proposition \ref{prop:m_odd}. The dots in red represent all the oscillators of some maximal set $S$ with $d(\phi^*,S)<\frac{4\pi}{m}$ }\label{fig:thm_m_odd}
\end{figure}

Now we show the instability of any equilibria with isotropy $(S_{k_1}\times S_{k_2}\times\dots\times S_{k_{l_B}})^m\rtimes Z_m$ for $m$ odd and greater or equal to $7$.

\begin{proposition}[Instability for $m\geq 7$ and odd]\label{prop:m_odd}
Suppose $f\in\mathcal F_{\frac{\pi}{2}}$ with $f$ concave in $[0,\pi]$ and $f'$ concave in $[-\frac{\pi}{2},\frac{\pi}{2}]$, then for all $m=2k+1$ with $k\geq3$ the equilibria $\phi^*$s with isotropy $(S_{k_1}\times S_{k_2}\times\dots\times S_{k_{l_B}})^m\rtimes Z_m$ are unstable.
\end{proposition}

The proof of Proposition \ref{prop:m_odd} also uses our cut condition to show instability, but with a different cut induced by the partition $P=(S,V(G)\backslash S)$ of $V(G)$ where $S$ is set to the a maximal subset of $V(G)$ such that $d(\phi,S)<\frac{4\pi}{m}$, see Figure \ref{fig:thm_m_odd} for an illustration of $P$. Notice that any of these partitions will include all the oscillators of two consecutive blocks of every constellation.
The details of the proofs are rather technical and are relegated to Appendix \ref{ap:proof_m_odd}.

\section{Effect of Delay}
\label{sec:delay}
 Once delay is introduced to the system of coupled oscillators, the problem becomes fundamentally harder. For example, for pulse-coupled oscillators, the reception of a pulse no longer gives accurate information about the relative phase difference $\Delta\phi_{ij}=\phi_j-\phi_i$ between the two interacting oscillators. Before, at the exact moment when $i$ received a pulse from $j$, $\phi_j$ was zero and the phase difference was estimated locally by $i$ as $\Delta\phi_{ij}=-\phi_i$. However, now when $i$ receives the pulse, the difference becomes $\Delta\phi_{ij}=-\phi_i-\psi_{ij}$. Therefore, the delay propagation acts as an error introduced to the phase difference measurement and unless some information is known about this error, it is not possible to predict the behavior. Moreover, as we will see later, slight changes in the distribution can produce nonintuitive behaviors.

 Even though it may not be satisfactory for some applications, many existing works choose to ignore delay. (see for e.g., \cite{Mirollo&Strogatz90SynchronizationofPulseCoupled, Lucarelli&Wang2004decentralizedsynchronization, Monzon&Paganini2005gobalconsidarations}). That is mainly for analytical tractability. On the other hand, when delay is included \cite{Papachristodoulou&Jadbabie2006SynchronizationinOscillator} the studies concentrate on finding  bounds on delay that maintain stability.

In this section we study how delay can {\it change} the stability in a network of weakly coupled oscillators.
A new framework to study these systems with delay will be set up by constructing an equivalent non-delayed system that has the same behavior as the original one in the continuum limit. We then further use this result to show that large heterogeneous delay can help reach synchronization, which is a bit counterintuitive and significantly generalizes previous related studies \cite{Izhikevich99WeaklyPulseCoupled, Vreeswijk&Abbott&Ermentrout94wheninhibition, Gerstner1996RapidPhaseLocking}. We will assume complete graph to simplify notation and exposition although the results can be extended for a boarder class of densely connected networks.

The contribution of this section is two fold. First, it  improves the understanding of the effect of delays in networks of coupled oscillators. And second, it opens new possibilities of using delay based mechanisms  to increase the region of attraction of the in-phase equilibrium set.
We shall build on existing arguments such as mean field approximation \cite{Kuramoto84chemicaloscillations} and Lyapunov stability theory \cite{Monzon&Paganini2005gobalconsidarations,Jadbabaie&Motee&Barahona2004stabilityofkuramoto} while looking at the problem from a different perspective.


\subsection{Mean Field Approximation}
Consider the case where the coupling between oscillators is all to all and identical ($\mathcal N_i=\mathcal N\backslash \{i\}$, $\forall i\in\mathcal N$ and $f_{ij}=f$ $\forall i,j$). And assume the phase lags $\psi_{ij}$ are randomly and independently chosen from the same distribution with probability density $g(\psi)$. By letting $N\rightarrow +\infty$ and $\varepsilon \rightarrow 0$ while keeping $\varepsilon N=:\bar{\varepsilon}$ a constant, \eqref{eq:phase_dynamics} becomes
\begin{equation}\label{eq:velocity_field}
v(\phi,t):=\omega + \bar{\varepsilon} \int_{-\pi}^{\pi}\int_0^{+\infty} f(\sigma - \phi - \psi)g(\psi)\rho(\sigma,t) d\psi d\sigma,
\end{equation}
where $\rho(\phi,t)$ is a time-variant normalized phase distribution that keeps track of the fraction of oscillators with phase $\phi$ at time $t$, and $v(\phi,t)$ is the velocity field that expresses the net force that the whole population applies to a given oscillator with phase $\phi$ at time $t$. Since the number of oscillators is preserved at any time, the evolution of $\rho(\phi,t)$ is governed by the continuity equation
\begin{equation}\label{eq:continuity}
\frac{\partial\rho}{\partial t} + \frac{\partial}{\partial\phi}(\rho v) = 0
\end{equation}
 with the boundary conditions $\rho(0,t)\equiv\rho(2\pi,t)$.
Equations \eqref{eq:velocity_field}-\eqref{eq:continuity} are not analytically solvable in general. Here we propose a new perspective that is inspired by the following observation. 

Consider the non-delayed system of the form 
\begin{equation}\label{eq:non_lagged_dynamics}
\dot{\phi}_i=\omega + \varepsilon\sum_{j\in \mathcal N_i}H(\phi_j - \phi_i),
\end{equation}
where
\begin{equation}\label{eq:convolution}
H(\theta)=f\ast g(\theta)=\int_0^{+\infty}f(\theta-\psi)g(\psi) d\psi
\end{equation}
is the convolution between $f$ and $g$.

By the same reasoning of \eqref{eq:velocity_field} it is easy to see
that the limiting velocity field of \eqref{eq:non_lagged_dynamics}
is
\begin{align*}
&v_H(\phi,t)  =\omega + \bar{\varepsilon}\int_{0}^{2\pi}H(\sigma -\phi)\rho(\sigma,t)d\sigma \\
&                =\omega + \bar{\varepsilon}\int_{0}^{2\pi}\left(\int_{0}^{+\infty}f((\sigma-\phi)-\psi)g(\psi) d\psi \right) \rho(\sigma,t)d\sigma \\
&                =\omega + \bar{\varepsilon} \int_{0}^{2\pi}\int_{0}^{+\infty}f(\sigma-\phi-\psi)g(\psi)\rho(\sigma,t) d\psi d\sigma     \\&     =v(\phi,t)
\end{align*}
where in the first and the third steps we used \eqref{eq:convolution} and \eqref{eq:velocity_field} respectively.
Therefore, \eqref{eq:phase_dynamics} and \eqref{eq:non_lagged_dynamics} {\bf have the same continuum limit.}

%

\begin{remark}
Although \eqref{eq:non_lagged_dynamics} is quite different from \eqref{eq:phase_dynamics}, both systems behave exactly the same in the continuum limit.  Therefore, as $N$ grows, \eqref{eq:non_lagged_dynamics} starts to become a good approximation of \eqref{eq:phase_dynamics} and therefore can be analyzed to understand the behavior of \eqref{eq:phase_dynamics}.
\end{remark}

\begin{figure}[htp]
\centering
\includegraphics[width=\columnwidth]{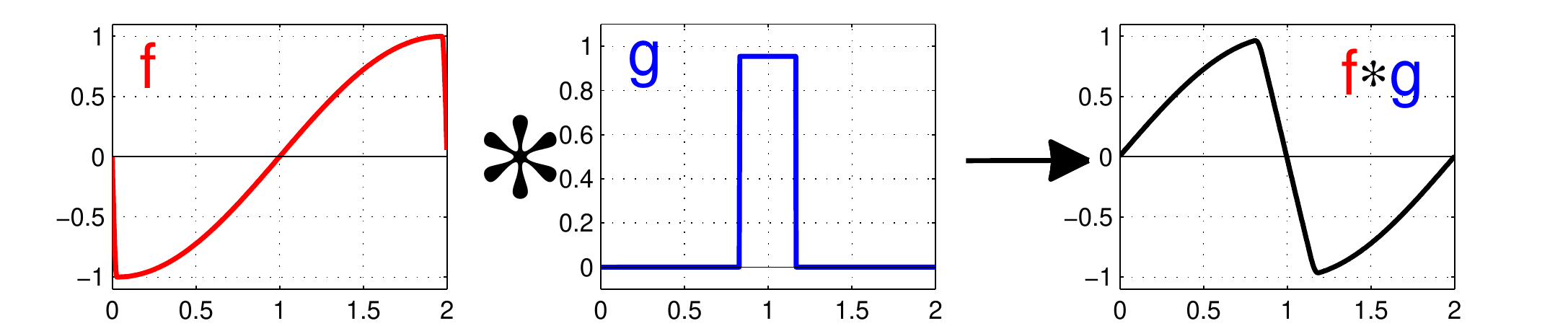}
\caption{{Effect of delay in coupling shape}}\label{fig:conv}
\end{figure}

Figure \ref{fig:conv} shows how, the underlying delay (in this case the delay distribution) determines what type of coupling (attractive or repulsive) produces synchronization. The original function $f$ produces repulsive coupling, whereas the corresponding $H$ is attractive. In fact, as we will soon see, the distribution of delay not only can qualitatively affect the type of coupling but also can change the stability of certain phase-locked limit cycles.

 We now study two example to illustrate how this new approximation can provide significant information about performance and stability of the original system. We also provide numerical simulations to verify our predictions.

\subsection{Kuramoto Oscillators}

We start by studying an example in the literature \cite{Watanabe&Strogatz1994constantsofmotions} to demonstrate how we can use the previous equivalent non-delayed formulation to provide a better understanding of systems of coupled oscillators with delay. When  $f(\theta)=K\sin(\theta)$,  $H(\theta)$ can be easily calculated:
\begin{align*}
&H(\theta)=\int_{0}^{+\infty}K\sin(\theta - \psi)g(\psi)d\psi \\&=K\int_{0}^{+\infty}\Im[e^{i(\theta - \psi)}g(\psi)]d\psi
=K\Im[e^{i\theta}\int_{0}^{+\infty}e^{-i\psi}g(\psi)d\psi] \\&=K\Im\left[e^{i\theta}Ce^{-i\xi}\right] = KC\sin(\theta - \xi)
\end{align*}
where $\Im$ is the imaginary part of a complex number, i.e. $\Im[a+ib]=b$. The values of $C>0$ and $\xi$ are calculated using the identity
\[
    Ce^{i\xi}=\int_{0}^{+\infty}e^{i\psi}g(\psi)d\psi.
\]
This complex number, usually called ``order parameter'', provides a measure of how the phase-lags are distributed within the unit circle. It can also be interpreted as the center of mass of the lags $\psi_{ij}$'s when they are thought of as points ($e^{i\psi_{ij}}$) within the unit circle $\mathds S^1$. Thus, when $C\approx 1$, the $\psi_{ij}$'s are mostly concentrated around $\xi$. When $C\approx 0$, the delay is distributed such that $\sum_{ij} e^{i\psi_{ij}}\approx 0$.

In this example, \eqref{eq:non_lagged_dynamics} becomes
\begin{equation}\label{eq:sine_approximation}
\dot{\phi}_i= \omega + \varepsilon KC\sum_{j\in\mathcal{ N}_i}\sin(\phi_j- \phi_i -\xi).
\end{equation}

Here we see how the distribution of $g(\psi)$ has a direct effect on the dynamics. For example, when the delays are heterogeneous enough such that $C\approx 0$, the coupling term disappears and therefore makes synchronization impossible. A complete study of the system under the context of superconducting Josephson arrays was performed \cite{Watanabe&Strogatz1994constantsofmotions} for the complete graph topology. There the authors characterized the condition for in-phase synchronization in terms of $K$ and $Ce^{i\xi}$. More precisely, when $KCe^{i\xi}$ is on the right half of the plane  ($KC\cos(\xi)>0$), the system almost always synchronizes. However, when $KCe^{i\xi}$ is on the left half of the plane ($KC\cos(\xi)<0$), the system moves towards an incoherent state where all of the oscillators' phases spread around the unit circle such that its order parameter, i.e. $
    \frac{1}{N}\sum_{l=1}^{N} e^{i\phi_l},
$
becomes zero.


%
%

\begin{figure}[htp]
\centering
\includegraphics[width=.35\columnwidth]{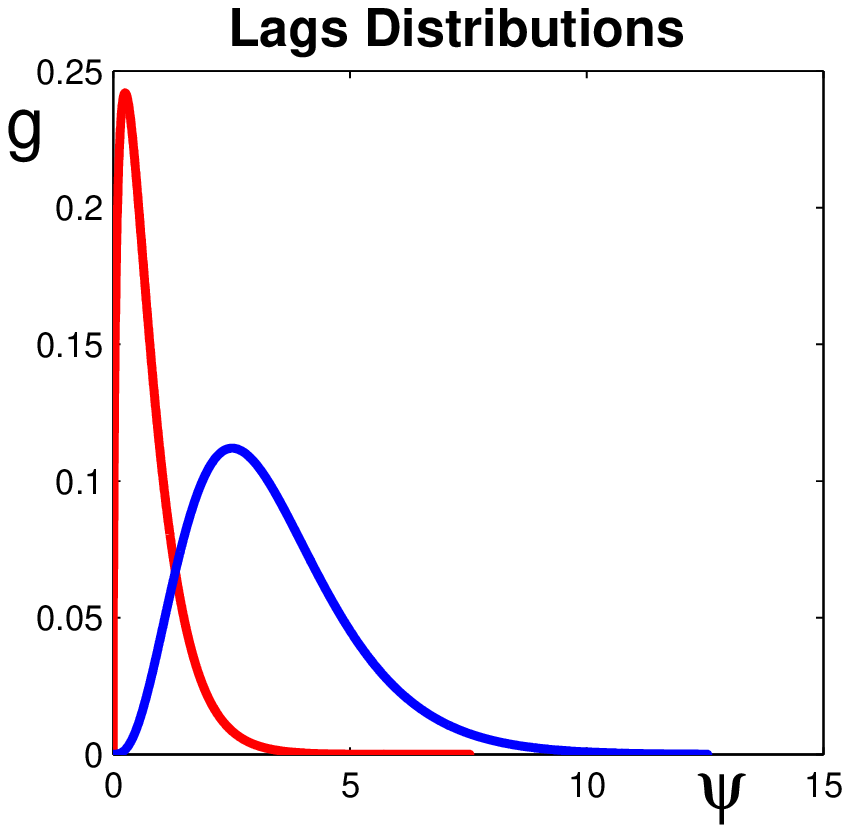}
\includegraphics[width=.35\columnwidth]{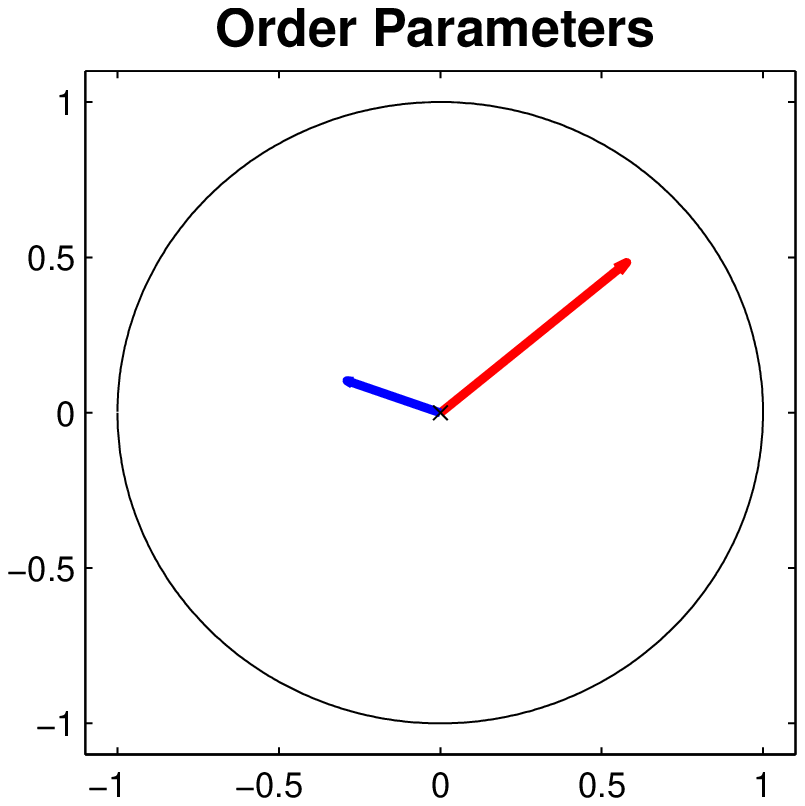}
\caption{{ Delay distributions and their order parameter $Ce^{i\xi}$} }
\label{fig:delay_distributions}
\end{figure}
\begin{figure}[htp]
\centering
\includegraphics[width=.85\columnwidth]{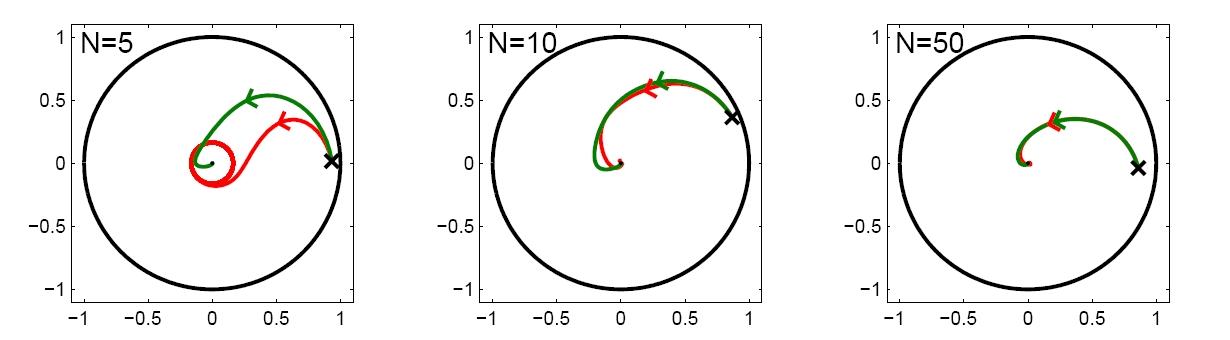}
\includegraphics[width=.85\columnwidth]{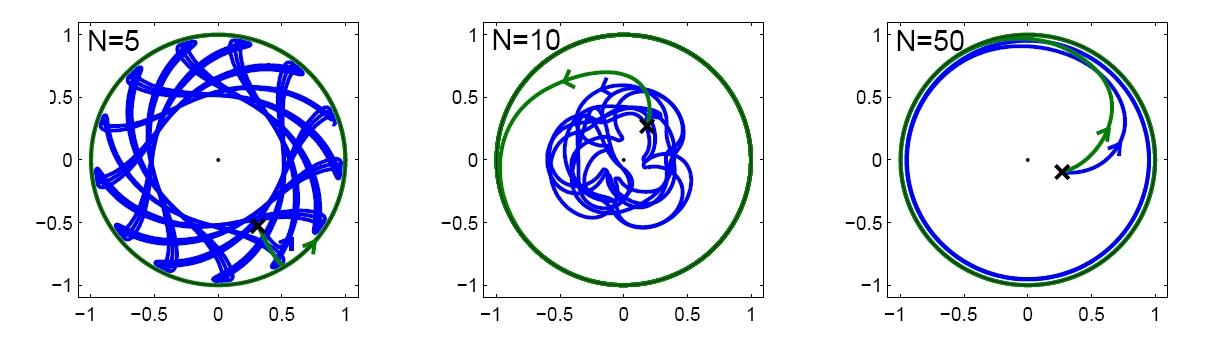}
\caption{{ Repulsive sine coupling with heterogeneous delays }}
\label{fig:sine_example}
\end{figure}


We now provide simulation results to illustrate how \eqref{eq:sine_approximation} becomes a good approximation of the original system when $N$ is large enough. We simulate the original repulsive ($K<0$) sine-coupled system with heterogeneous delays and its corresponding approximation \eqref{eq:sine_approximation}.  Two different delay distributions, depicted in Figure \ref{fig:delay_distributions}, were selected such that their corresponding order parameter lie in different half-planes.

The same simulation is repeated for $N=5,10,50$. Figure \ref{fig:sine_example} shows that when $N$ is small, the phases' order parameter of the original system (in red/blue) draw a trajectory which is completely different with respect to its approximation (in green). However, as $N$ grows, in both cases the trajectories become closer and closer. Since $K<0$, the trajectory of the system with wider distribution ($C\cos \xi<0$) drives the order parameter towards the boundary of the circle, i.e., \textbf{heterogeneous delay leads to homogeneous phase}.


\subsection{Effect of Heterogeneity}

We now explain a more subtle effect that heterogeneity can produce. Consider the system in \eqref{eq:non_lagged_dynamics} where $H$ odd and continuously differentiable. Then, from Section \ref{sec:topology&coupling}, all of the oscillators eventually end up running at the same speed $\omega$ with fixed phase difference such that the sum $\sum_{i\in \mathcal N_i} H(\phi_j-\phi_i)$ cancels $\forall i$.
Moreover, we can apply  \eqref{eq:cut_instability_cond} to assess the stability of these orbits. Therefore, if we can find a cut $C$ of the network such that
$
\sum_{ij\in C} H'(\phi^*_j-\phi^*_i) < 0,
$
the phase-locked solution will be unstable.

Although this condition is for non-delayed phase-coupled oscillators, the result of this section allows us to translate it for systems with delay. Since $H$ is the convolution of the coupling function $f$ and the delay distribution function $g$, we can obtain $H'(\phi^*_j-\phi^*_i)<0$, even when $f'(\phi^*_j-\phi^*_i)>0$. This usually occurs when the convolution widens the region with a negative slope of $H$. See Figure \ref{fig:conv} for an illustration of this phenomenon.

\begin{figure}[htp]
\centering
\includegraphics[width=.8\columnwidth]{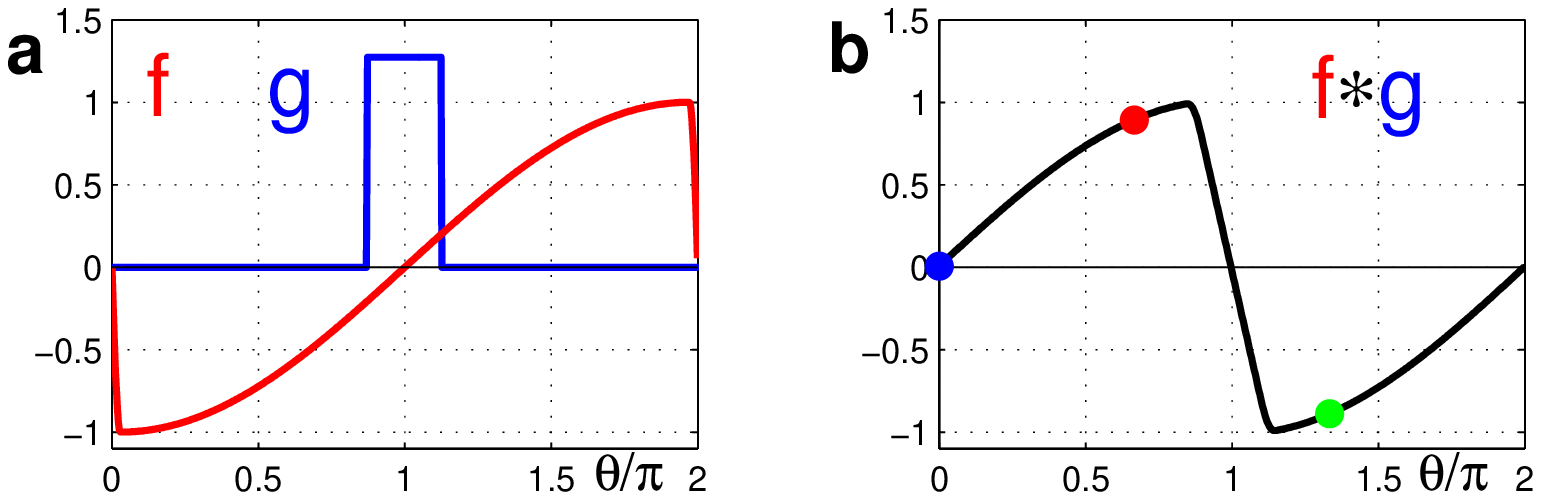}
\includegraphics[height=.45\columnwidth,width=.8\columnwidth]{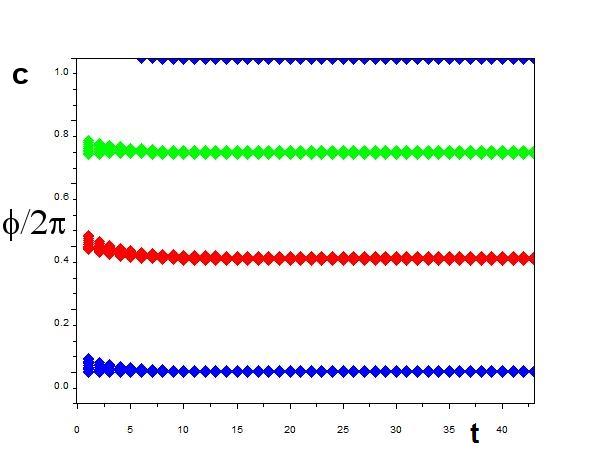}
\caption{{ Pulse-coupled oscillators with delay: Stable equilibrium } }
\label{fig:stable}
\end{figure}

\begin{figure}[htp]
\centering
\includegraphics[width=.8\columnwidth]{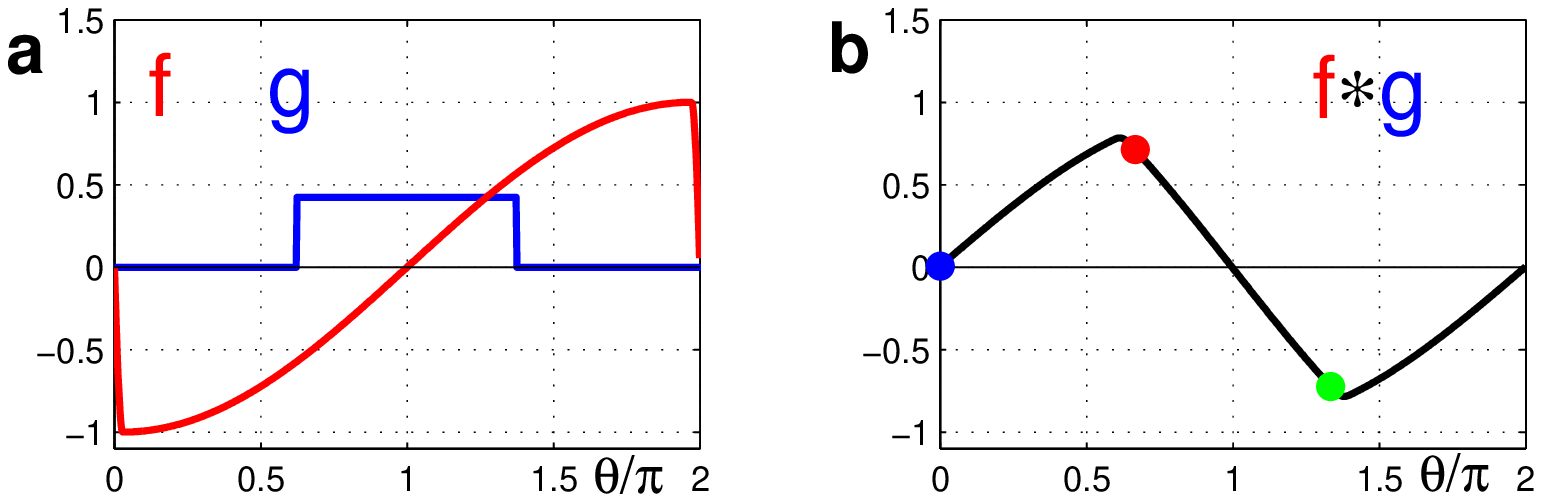}
\includegraphics[height=.45\columnwidth,width=.8\columnwidth]{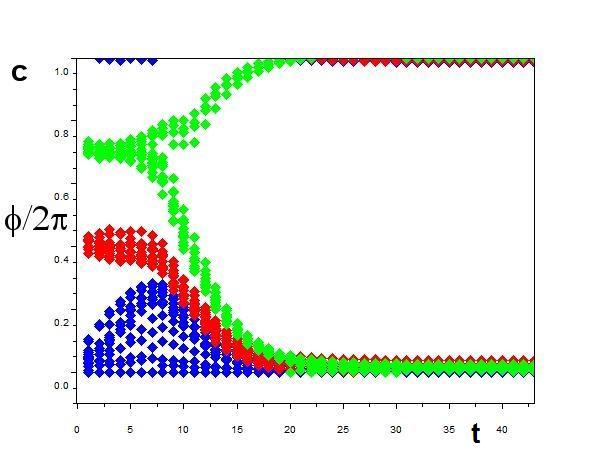}
\caption{{ Pulse-coupled oscillators with delay: Unstable equilibrium }}
\label{fig:unstable}
\end{figure}

Figures \ref{fig:stable} and \ref{fig:unstable} show two simulation setups of 45 oscillators {\bf pulse-coupled} all to all. The initial state is close to a phase locked configuration formed of three equidistant clusters of 15 oscillators each. The shape of the coupling function $f$ and the phase lags distributions are shown in part a. We used \eqref{eq:f_kappa} to implement the corresponding pulse-coupled system~\eqref{pc_model}. While $f$ is maintained unchanged between both simulations, the distribution $g$ does change.
 Thus, the corresponding $H=f\ast g$ changes as it can be seen in part b; the blue, red, and green dots correspond to the speed change induced in an oscillator within the blue cluster by oscillators of each cluster. Since all clusters have the same number of oscillators, the net effect is zero. In part c the time evolution of oscillators' phases relative to the phase of a blue cluster oscillator are shown. Although the initial conditions are exactly the same, the wider delay distribution on Figure \ref{fig:unstable} produces negative slope on the red and green points of part b, which destabilizes the clusters and drives oscillators toward in-phase synchrony.

Finally, we simulate the same scenario as in Figures \ref{fig:stable} and \ref{fig:unstable} but now changing $N$ and the standard deviation, i.e. the delay distribution width. Figure \ref{fig:sinc_prob} shows the computation of the  synchronization probability vs. standard deviation. The dashed line denotes the minimum value that destabilizes the equivalent system. As $N$ grows, the distribution shape becomes closer to a step, which is the expected shape in the limit. It is quite surprising that as soon as the equilibrium is within the region of $H$ with negative slope, the equilibrium becomes unstable as the theory predicts.

\begin{figure}[htp]
\centering
\includegraphics[height=.8\columnwidth]{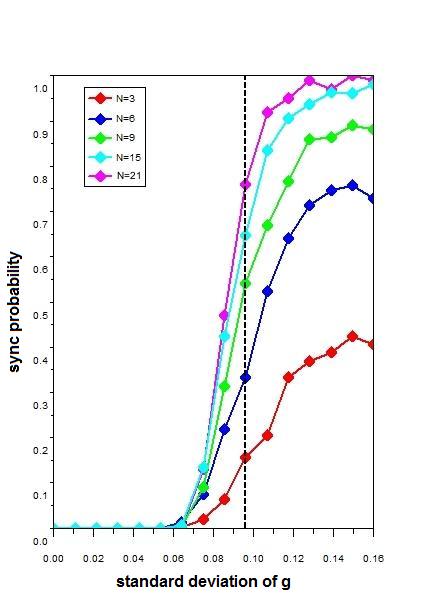}
\caption{{Pulse-coupled oscillators with delay: Synchronization probability }}\label{fig:sinc_prob}
\end{figure}


\section{Conclusion}
\label{sec:conclusion}

This paper analyzes the dynamics of identical weakly coupled oscillators while relaxing several classical assumptions on coupling, delay and topology. Our results provide global synchronization guarantee for a wide range of scenarios. There are many directions that can be taken to further this study. For example, for different topologies, to guarantee global in-phase synchronization, how does the requirement on coupling functions change? Another specific question is to complete the proof in Section \ref{sec:coupling} for the cases when $m=1,3,5$. Finally, it will be of great interest if we can apply results and techniques in this paper to a wide range of applications such as transient stability analysis of power networks and clock synchronization of computer networks.
\vspace{0.09in}

\noindent {\bf Acknowledgments:} The authors thank Steven H. Strogatz of Cornell for useful discussions.

\appendix

\section {Proof of Proposition \ref{prop:m_odd} }\label{ap:proof_m_odd}

As in Proposition \ref{prop:m_even} we will use our cut condition to show the instability of $\phi^*$. Thus, we define a partition $P=(S,V(G)\backslash S)$ of $V(G)$ by taking $S$ to be a maximal subset of $V(G)$ such that $d(\phi,S)<\frac{4\pi}{m}$, see Figure \ref{fig:thm_m_odd} for an illustration of $P$. Notice that any of these partitions will include all the oscillators of two consecutive blocks of every constellation.

Instead of evaluating the total sum of the weights in the cut we will show that the sum of edge weights of the links connecting the nodes of one constellation in $S$ with the nodes of a possibly different constellation in $V(G)\backslash S$ is negative. In other words, we will focus on showing
\begin{equation}\label{eq:l1l2}
\sum_{ij\in \mathcal K_{l_1l_2}} f'(\phi_j^*-\phi_i^*)<0
\end{equation}
where $\mathcal K_{l_1l_2}=\{ij:i\in C_{l_1}\cap S, j\in C_{l_2} \cap V(G)\backslash S\}$.

Given any subset of integers $J$, we define
\[
g_m^{J}(\delta)=g_m(\delta) - \sum_{j\in J} f(\frac{2\pi}{m}j+\delta).
\]

Then, we can rewrite \eqref{eq:l1l2} as
\begin{align}
\sum_{ij\in \mathcal K_{l_1l_2}} f'(\phi_j^*&-\phi_i^*)= \nonumber\\=& (g_m^{\{0,1\}})'(\delta_{l_1l_2}) + (g_m^{\{-1,0\}})'(\delta_{l_1l_2})
\nonumber \\
=&2g_m'(\delta_{l_1l_2}) -f'(\delta_{l_1l_2} + \frac{2\pi}{m} )  - 2f'( \delta_{l_1l_2} ) \nonumber\\&- f'( \delta_{l_1l_2} - \frac{2\pi}{m} )\label{eq:l1l2b}
\end{align}
where  $\delta_{l_1l_2}\in[0,\frac{2\pi}{m}]$ is the phase shift between the two constellations.
Then, if we can show that for all $\delta\in [0,\frac{2\pi}{m}]$ \eqref{eq:l1l2b} is less than zero then for any values of $l_1$ and $l_2$ we will have \eqref{eq:l1l2} satisfied.

Since $f$ is odd and even around $\frac{\pi}{2}$, $f'$ is even and odd around $\frac{\pi}{2}$ and  $g_m'(\delta)$ can be rewritten as
\begin{align*}
g_m'(\delta) &= f'(\delta)\\
 +& \sum_{1\leq \norm{j}\leq\lfloor \frac{k}{2}\rfloor} \left\{ f'(\delta + \frac{2\pi}{m}j)-f'(\delta -sgn(j)\frac{\pi}{m} + \frac{2\pi}{m}j) \right\}\\
-&\left[f'(\delta +\frac{\pi}{m}k)+f'(\delta -\frac{\pi}{m}k)\right]\ones_{[k\text{ odd}]}
\end{align*}
where $\ones_{[k\text{ odd}]}$ is the indicator function of the event $[k\text{ odd}]$, the sum is over all the integers $j$ with $1\leq\norm{j}\leq\lfloor\frac{k}{2}\rfloor$ and $k=\frac{m-1}{2}$


The last term only appears when $k$ is odd and in fact it is easy to show that it is always negative as the following calculation shows:
\begin{align*}
&-f'(\delta +\frac{\pi}{m}k)-f'(\delta -\frac{\pi}{m}k)=\\&=- f'(\frac{\pi}{m}k+ \delta)-f'(\frac{\pi}{m}k -\delta )\\&=-f'( \frac{\pi}{2} -\frac{\pi}{2m}+ \delta)-f'(\frac{\pi}{2} -\frac{\pi}{2m} - \delta )\\
&=f'( \frac{\pi}{2} -\delta +\frac{\pi}{2m})-f'(\frac{\pi}{2} -\delta -\frac{\pi}{2m})\\&= f'( \theta  )- f'(\theta - \phi) <0
\end{align*}
where in step one we used the fact of $f'$ being even, in step two we used $k=\frac{m-1}{2}$ and in step three we use $f'$ being odd around $\frac{\pi}{2}$. The last step comes from substituting $\theta=\frac{\pi}{2} -\delta +\frac{\pi}{2m}$, $\phi=\frac{\pi}{m}$ and apply Claim \ref{claim:monotonicity}, since for $m\geq7$  we have $0\leq\theta-\phi<\theta\leq\pi$.

Then it remains the show that the terms of the form $f'(\delta + \frac{2\pi}{m}j)-f'(\delta -sgn(j)\frac{\pi}{m} + \frac{2\pi}{m}j)$ are negative for all $j$ s.t. $1\leq \norm{j}\leq \lfloor \frac{k}{2}\rfloor$. This is indeed true when $j$ is positive since for all $\delta \in [0,\frac{2\pi}{m}]$ we get
\[
0\leq \delta  -\frac{\pi}{m} + \frac{2\pi}{m}j < \delta + \frac{2\pi}{m}j \leq \pi, \text{ for } 1\leq j \leq \lfloor \frac{k}{2}\rfloor
\]
and thus we can apply again Claim \ref{claim:monotonicity}.

When  $j$ is negative there is one exception in which Claim \ref{claim:monotonicity} cannot be used since
\[
-\pi\leq \delta + \frac{2\pi}{m}j   < \delta + \frac{2\pi}{m}j +\frac{\pi}{m}\leq 0, \forall \delta\in[0,\frac{2\pi}{m}]
\]
only holds for $-\lfloor \frac{k}{2} \rfloor \leq j \leq -2$. Thus the term corresponding to $j=-1$ cannot be directly eliminated.

Then, by keeping only the terms of the sum with $j=\pm1$, $g'_m$ is strictly upper bounded for all $\delta\in[0,\frac{2\pi}{m}]$ by
\begin{align}
g_m'(\delta) < 
&f'(\delta) + f'(\delta -\frac{2\pi}{m}) -f'(\delta -\frac{\pi}{m})  \nonumber\\&+ f'(\delta +\frac{2\pi}{m}) -f'(\delta +\frac{\pi}{m})\label{eq:boundongmp}
\end{align}

Now substituting \eqref{eq:boundongmp} in \eqref{eq:l1l2b} we get
\begin{align*}
&\sum_{ij\in \mathcal K_{l_1l_2}} f'(\phi_j^*-\phi_i^*)
\\ 
&<  f'(\delta -\frac{2\pi}{m}) - 2f'(\delta -\frac{\pi}{m})+ f'(\delta +\frac{2\pi}{m})-2f'(\delta +\frac{\pi}{m})
\end{align*}
\begin{equation*}
\begin{aligned}
&\leq f'(\delta -\frac{2\pi}{m}) - 2f'(\delta -\frac{\pi}{m})  -f'(\delta +\frac{\pi}{m})
\\ 
&\leq f'(\delta -\frac{2\pi}{m}) - 2f'(\delta -\frac{\pi}{m}) 
\end{aligned}
\end{equation*}
where in the last step we used the fact that for $m\geq6$ and $\delta\in[0,\frac{2\pi}{m}]$, $f'(\delta +\frac{\pi}{m})\geq 0$.

Finally, since for $\delta\in[0,\frac{2\pi}{m}]$ $f'(\delta -\frac{2\pi}{m})$ is strictly increasing and $f'(\delta -\frac{\pi}{m})$ achieves its minimum for $\delta\in\{0,\frac{2\pi}{m}\}$, then
\[
 f'(\delta -\frac{2\pi}{m}) - 2f'(\delta -\frac{\pi}{m})\leq f'(0) - 2f'(\frac{\pi}{m})\leq 0
\]
where the last inequality follow from Claim \ref{claim:concavity}.

Therefore, for all $m$ odd greater or equal to $7$ we obtain
\[
\sum_{ij\in \mathcal K_{l_1l_2}} f'(\phi_j^*-\phi_i^*)< f'(0) - 2f'(\frac{\pi}{m})\leq 0
\]
and since this result is independent on the indices $l_1$, $l_2$, then
\begin{align*}
&\sum_{ij\in C(S,V(G)\backslash S)} f'(\phi_j^*-\phi_i^*)  \\&\quad\quad\quad\quad= \sum_{l_1=1}^{l_B}\sum_{l_2=1}^{l_B}\sum_{ij\in \mathcal K_{l_1l_2}} f'(\phi_j^*-\phi_i^*)<0
\end{align*}
and thus $\phi^*$ is unstable.

\bibliographystyle{ieeetr}
\bibliography{mallada}

\end{document}